\documentclass[12pt]{amsart}

\pdfoutput=1

\usepackage{latexsym}
\usepackage[centertags]{amsmath}
\usepackage{amsfonts}
\usepackage{amssymb}
\usepackage{amsthm}
\usepackage{newlfont}
\usepackage{graphics}
\usepackage{color}
\usepackage[usenames,dvipsnames]{xcolor}

\usepackage{parskip}
\usepackage{tikz}
\usepackage{eqnarray,verbatim}
\usepackage{mathrsfs}
\usepackage{enumerate}
\usepackage[T1]{fontenc}
\usetikzlibrary{patterns,arrows,decorations.pathreplacing}
\usepackage{epsfig}
\usepackage{graphicx}

\newcommand{\tref}[1]{Theorem \ref{#1}}
\newcommand{\lref}[1]{Lemma \ref{#1}}
\newcommand{\fref}[1]{Figure \ref{fig:#1}}
\newcommand{\taref}[1]{Table \ref{table:#1}}
\newcommand{\qtn}[1]{[#1]_{q,t}}

\newcommand{\fillll}[3]{\node at (#1-.5,#2-.5) {$#3$};}

\usepackage{graphicx}

\usepackage{wrapfig}

\usepackage{float} 

\newtheorem{theo}{Theorem}

\newtheorem{lem} [theo]{Lemma}

\makeatletter \@addtoreset{equation}{section}
\@addtoreset{theo}{}\makeatother

\newcommand\qbinom[2]{\left[#1 \atop #2  \right]_q}

\def\N{\mathbb{N}}
\def\Z{\mathbb{Z}}

\newcommand{\el}{\ell}
\newcommand{\la}{\lambda}

\def\TH{\tilde{H}}

\newcommand\scalar[2]{\langle #1, \; #2  \rangle}
\newcommand{\drawtab}[3]{\begin{tikzpicture}[scale =.4] 
\draw[help lines] (0,0) grid (#1 + #2 + #3 + 1, 1);
\ifnum #1>0 
\foreach \dir in {1,...,#1}{
	\node at (\dir-.5,.5) {\footnotesize$1$};
};
\fi 
\ifnum #2>0 
\foreach \dir in {1,...,#2}{
	\node at (#1+\dir-.5,.5) {\footnotesize$2$};
};
\fi 
\ifnum #3>0 
\foreach \dir in {1,...,#3}{
	\node at (#1+#2+\dir-.5,.5) {\footnotesize$\overline{2}$};
};
\fi 
\node at (#1+#2+#3+.5,.5) {\footnotesize$\widehat{2}$};
\end{tikzpicture}}

\newcommand{\tabw}[4]{\begin{tabular}{c}
	\drawtab{#1}{#2}{#3}\\$#4$
\end{tabular}}

\newcommand{\tabwr}[4]{\begin{tabular}{c}
	\drawtab{#1}{#2}{#3}\\ {\color{red}$#4$}
\end{tabular}}

\title{On the Schur Positivity of $\Delta_{e_2} e_n[X]$}

\author{Dun Qiu$^1$, Jeffrey B. Remmel, Emily Sergel$^2$, and Guoce Xin$^3$}

\address{
	$^1$Department of Mathematics, University of California San Diego, La Jolla, CA 92093-0112, USA\\
	$^2$Department of Mathematics, University of Pennsylvania, Philadelphia, PA 19104-6395, USA\\
	$^3$School of Mathematical Sciences, Capital Normal University, Beijing 100048, PR China}

\email{
	$^1$\texttt{duqiu@ucsd.edu} \\
	\&\small \, $^2$\texttt{esergel@math.upenn.edu}
\\  \&\small \, $^3$\texttt{guoce.xin@gmail.com}}

\date{\today}

\begin{document}

\begin{abstract}
Let $\mathbb{N}$ denote the set of non-negative integers. Haglund, Wilson, and the second author have conjectured that the coefficient of any Schur function $s_\la[X]$ in $\Delta_{e_k} e_n[X]$ is a polynomial in $\mathbb{N}[q,t]$. We present four proofs of a stronger statement in the case $k=2$; We show that the coefficient 
of any Schur function $s_\lambda[X]$ in 
$\Delta_{e_2} e_n[X]$ has a positive expansion in terms of $q,t$-analogs.
\end{abstract}

\maketitle

\vspace{-5mm}

\section{Introduction} \label{s-intro}

Let $\Lambda$ denote the ring of symmetric functions with coefficients in $\mathbb{Q}(q,t)$. If $\mu$ is a partition of $n$, we shall write 
$\mu \vdash n$. Let $X=x_1+ \cdots +x_N$. 
The sets $\{e_{\mu}[X] : \mu \vdash n\}$, $\{s_{\mu}[X] : \mu \vdash n\}$
 and $\{\tilde{H}_{\mu}[X;q,t] : \mu \vdash n\}$ are the elementary, 
the Schur,  and the {(modified) Macdonald} symmetric function bases for $\Lambda^{(n)}$, the elements of $\Lambda$ that are homogeneous of degree $n$. Given a partition $\mu \vdash n$ and a cell $c$ in the Young diagram of $\mu$ (drawn in French notation), we set $a^{\prime}(c)$ and $\el^{\prime}(c)$ to be the number of cells in $\mu$ that are strictly to the left and strictly below $c$ in $\mu$, respectively. 
For example, if $\mu = (3,4,4,5)$ and $c$ is the cell pictured in 
Figure \ref{fig:leg}, then $a^{\prime}(c) =3$ is represented 
by the cells containing dots and $\el^{\prime}(c)=2$ is represented 
by the cells containing stars. 

\begin{figure}[H]
	\begin{center}
		\includegraphics[height=1.0in]{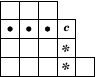}
		\caption{$a^{\prime}(c)$ and $\el^{\prime}(c)$ .}
		\label{fig:leg}
	\end{center}
\end{figure}

We set 
\begin{align*}
&B_{\mu}(q,t) = \sum_{c \in \mu} q^{a^{\prime}(c)} t^{\el^{\prime}(c)}, &T_{\mu}(q,t) = \prod_{c \in \mu} q^{a^{\prime}(c)} t^{\el^{\prime}(c)} .
\end{align*} 
Given any symmetric function $f \in \Lambda$, we define operators $\Delta_f$ and $\Delta^{\prime}_f$ on $\Lambda$ by their action on the Macdonald basis:
\begin{align*}
&\Delta_f \tilde{H}_{\mu}[X;q,t] = f[B_{\mu}(q,t)] \tilde{H}_{\mu}[X;q,t],
&\Delta^{\prime}_f  \tilde{H}_{\mu}[X;q,t] = 
f[B_{\mu}(q,t)-1] \tilde{H}_{\mu}[X;q,t].
\end{align*}
Here, we have used the notation that, for a symmetric function $f$ and a sum $A = a_1 + \ldots + a_N$ of monic monomials, $f[A]$ is equal to the specialization of $f$ at $x_1=a_1, \ldots, x_N=a_N$, where the remaining variables are set equal to zero. We also set $\nabla = \Delta_{e_n}$ as an operator on $\Lambda^{(n)}$. Note that, by definition, for any $1 \leq k \leq n$,
\begin{align}
\Delta_{e_k} e_n[X] = \Delta^{\prime}_{e_k + e_{k-1}} e_n[X] = \Delta^{\prime}_{e_k} e_n[X] + \Delta^{\prime}_{e_{k-1}} e_n[X] .
\end{align}
Furthermore, for any $k > n$, $\Delta_{e_k} e_n[X] = \Delta^{\prime}_{e_{k-1}} e_n[X] = 0$. Therefore $\Delta_{e_n} e_n[X] = \Delta^{\prime}_{e_{n-1}} e_n[X]$.

In \cite{HRW}, Haglund, Remmel, and Wilson conjectured 
a combinatorial interpretation of the coefficients that appear in the expansion 
of $\Delta_{e_k} e_n[X]$ in terms of the fundamental quasi-symmetric 
functions. Their conjecture is now referred to as the $\Delta$-conjecture. 
They also conjectured that coefficients in 
the Schur function expansion of $\Delta_{e_k} e_n[X]$ are  
polynomials in $q$ and $t$ with non-negative 
integer coefficients. There are two cases that are known. 
Namely, when $k =n$, then Haiman \cite{Haiman} proved that 
$\Delta_{e_{n}} e_n[X] = \nabla e_n[X] $ is the Frobenius 
image of the character generating 
function of the ring of diagonal co-invariants.  Thus in this 
case, repesentation theory tells us that the 
coefficient of the Schur function $s_{\la}[X]$, 
$\langle \nabla e_n[X],s_{\la}[X]\rangle$, is a polynomial 
in $q$ and $t$ with non-negative integer coefficients.  
Also in this case, the so-called ``Shuffle conjecture'' 
of Haglund, Haiman, Loehr, Remmel, and Ulyanov \cite{HHLRU} 
gives a combinatorial interpretation of the coefficients 
that arise in the expansion of $\nabla e_n[X]$ in terms 
of fundamental quasi-symmetric functions.  The Shuffle conjecture 
was recently proved by Carlsson and Mellit \cite{CM}. 

The other known case is when $k=1$. In \cite{HRW}, the authors proved that 
\begin{equation}
\Delta_{e_{1}} e_n[X] = \sum_{m=0}^{\lfloor n/2 \rfloor} s_{2^m, 1^{n-2m}}[X] \sum_{p=m}^{n-m} [p]_{q,t}
\end{equation}
where $[n]_{q,t} = \frac{q^n-t^n}{q-t} = 
q^{n-1} + q^{n-2}t+ \cdots + qt^{n-2}+ t^{n-1}$ for $n \geq 0$. 

 
The main goal of this paper is to  
to give four different  proofs of the fact that $\Delta_{e_2}e_n[X]$ 
is Schur positive, i.e. for all $\lambda \vdash n$, 
$\langle \Delta_{e_2} e_n[X],s_\lambda[X] \rangle \in \mathbb{N}[q,t]$, 
in hopes that some of the ideas in those proofs can be 
adapted to prove the Schur positivity of $\Delta_{e_k}e_n[X]$ for 
$k \geq 3$.

All of our proofs start with the following result of Haglund \cite{Hag}.
\begin{lem} For all $n$, $d$, and symmetric functions $f[X]$,
\begin{align}
  \scalar{ \Delta_{e_{d-1}} e_n[X]}{ f[X]} =\scalar{ \Delta_{\omega f} \, e_d[X]}{ s_d[X]}.
\end{align}
\end{lem}

Let $\lambda$ be any partition of $n$. By setting $f=s_{\lambda}$, we have
\begin{align}
\scalar{ \Delta_{e_{d-1}} e_n[X]}{ s_\lambda} =\scalar{ \Delta_{s_{\lambda'}} e_d}{ s_d}.
\end{align}

The formula works nicely when $d$ is small, since we have explicit expansion of $e_d$ in terms of Macdonald polynomials.
In the case $d=2$ we have
$$ e_2[X] = \frac{1}{t{-}q} \TH_{1,1}[X;q,t] - \frac{1}{t{-}q} \TH_{2}[X;q,t].$$
This leads to
\begin{align*}
  \scalar{ \Delta_{e_1} e_n[X]}{ s_\lambda[X]} &= \scalar{ \Delta_{s_{\lambda'}} e_2[X]}{ s_2[X]}\\
&= \Big\langle \frac{1}{t{-}q}s_{\lambda'}[1{+}t] \TH_{1,1}[X;q,t] 
- \frac{1}{t{-}q}s_{\lambda'}[1{+}q]  \TH_{2}[X;q,t],\ s_2[X] \Big\rangle \\
&=\frac{1}{t{-}q}s_{\lambda'}[1{+}t]  - \frac{1}{t{-}q}s_{\lambda'}[1{+}q] ,
\end{align*}
which is easily seen to be an element of $\mathbb{N}(q,t)$.

In the case $d=3$, the expansion of $e_3$ leads to the following formula.
\begin{multline}\label{e-g-lambda-o}
g_\lambda:=\scalar{\Delta_{e_2} e_n[X]}{ s_\lambda[X]} = \\
 \frac{(t{-}q^2) s_{\lambda'}[1{+}t{+}t^2] -(q{+}t{+}1)(t{-}q) s_{\lambda'}[1{+}q{+}t] +(t^2{-}q) s_{\lambda'}[1{+}q{+}q^2] }{(t{-}q)(t^2{-}q)(t{-}q^2)}
\end{multline}
At first glance, this formula does not seem to be useful. 
Indeed, it is not immediately obvious that this quotient is a polynomial.

Our (chronologically) first approach to proving 
that $g_\lambda$ is in $\mathbb{N}[q,t]$ 
 is based on the following observations. 
\begin{enumerate}[i)]
\item If $\lambda'$ has more than three parts, then $g_\lambda=0$;
\item If we expand $s_{a,b,c}[x+y+z]$ as a quotient of alternates, then from the view of MacMahon partition analysis, one can easily see that 
the generating function $$\sum_{a\ge b \ge c \ge 0} g_{(a,b,c)'} u_1^a u_2^b u_3^c$$ is a rational function.
\item Hence, it might be easier to show that this generating function has only nonnegative coefficients.
\end{enumerate}
We succeeded in this approach by finding a proof that can be easily verified by computer, but it is too long to be printed. 
We will explain this approach in Section \ref{s-gen-fun}, but we will not include full details.

Our other approaches rely on the following alternative representation of $g_\lambda$.
\begin{lem}\label{l3} Let $\tau$ be the operation which switches $t$ and $q$. Then
\begin{align}\label{e-Delta-e2}
g_\lambda=  \scalar{ \Delta_{e_2} e_n[X]}{ s_\lambda[X]} =\frac{F_{\lambda'} - \tau F_{\lambda'}}{t{-}q}= \frac{id {-} \tau}{t{-}q} F_{\lambda'},
\end{align}
where $\tau F = F\big|_{q=t,t=q}$ and 
\begin{align}
  F_{\lambda'} = \frac{ s_{\lambda'}[1{+}t{+}t^2] - s_{\lambda'}[1{+}t{+}q] }{t^2{-}q}.\label{e-Flambda}
\end{align}
\end{lem}
\begin{proof}
By using the formula
$$(t{-}q)(1{+}q{+}t) 	\, = \, (t{-}q^2)-(q{-}t^2)		\, = \, (id{-}\tau)(t{-}q^2),$$
equation (\ref{e-g-lambda-o}) becomes
\begin{align*}
  \scalar{ \Delta_{e_2} e_n[X]}{ s_\lambda[X]} & = \frac{(id{-}\tau) (t{-}q^2) s_{\lambda'}[1{+}t{+}t^2] -(id{-}\tau)(t{-}q^2) s_{\lambda'}[1{+}q{+}t]}{(t{-}q)(t^2{-}q)(t{-}q^2)} \\
   &= \frac{1}{t{-}q} (id{-}\tau) \left( \frac{ s_{\lambda'}[1{+}t{+}t^2] - s_{\lambda'}[1{+}t{+}q] }{t^2{-}q}     \right).
\end{align*}
This is just the desired \eqref{e-Delta-e2}.
\end{proof}

We will show that $F_{\lambda'}$ is a polynomial that can be interpreted as a sum over semi-standard Young tableaux filled with numbers $0,1,2$.
From this formula, it is clear that $g_\lambda$ is in 
$\mathbb{Z}[q,t]$ where $\mathbb{Z} = \{0,\pm 1, \pm 2, \ldots \}$ is 
the set of integers. 

We present our second proof in Section \ref{s-injection}. We introduce 
new combinatorial objects, called ``enriched" semi-standard Young tableaux, to interpret the coefficients of $g_\lambda$. We then define 
an injection on these enriched tableaux which will allow us to prove 
that $g_\lambda$ is in $\mathbb{N}[q,t]$.

In Section \ref{s-direct-computation}, we present our third proof that $g_\lambda$ is in $\mathbb{N}[q,t]$. The proof 
in this section is a direct computation of $g_\lambda$ carried out 
by breaking $g_\lambda$ into a sum of terms where each term is easily seen to 
be a polynomial in $q$ and $t$ with non-negative coefficients. The 
advantage of this proof is that we can recursively produce explicit 
formulas for $g_\lambda$.

In Section \ref{s-gen-fun}, we shall expand our discussion 
of the generating function approach described above and describe 
an alternate way to analyze the resulting generating functions which is 
our fourth proof.

Finally in Section \ref{s-e3}, we give a formula of $\Delta_{e_3} e_n[X]$. 
However it is not clear how we can split up this formula into pieces 
which are easily seen to be polynomials in $q$ and $t$ with non-negative 
coefficients. Thus, the general problem of establishing the Schur-positivity of $\Delta_{e_d} e_n[X]$ seems to require new ideas. 


\section{Acknowledgements}

The authors would like to thank Professor Adriano Garsia for his invaluable contributions, including many productive discussions. The first, third, and fourth authors would also like to dedicate this paper to the memory of our second author, Professor Jeff Remmel, who passed away recently. His mentorship, collaboration, and friendship will be greatly missed.

The third author was partially supported by NSF grant DMS-1603681.

\section{Combinatorial Proof\label{s-injection}}
The idea is based on the following observation:
$$ \frac{(id {-}\tau) t^j q^i}{t{-}q} =-\frac{(id {-}\tau) t^i q^j}{t{-}q} \text{ and }\frac{(id {-}\tau) t^i q^j}{t{-}q}=(tq)^j [i{-}j]_{q,t}, \quad \mbox{ if } i\ge j.
$$
Thus if $F_{\lambda'}= \sum_{i,j} a_{i,j} q^i t^j$, we have
$$g_\lambda = \frac{(id {-}\tau) F_{\lambda'}}{t{-}q} = \sum_{i>j} (a_{i,j}{-}a_{j,i}) (tq)^j [i{-}j]_{q,t}.$$
To show that $g_\lambda \in \N[q,t]$, it is sufficient to show that
$a_{i,j}{-}a_{j,i}>0$ for every $i>j$. Note that this condition indeed shows the Schur-positivity of $g_\lambda$ in $q,t$-analogs, a stronger condition than $g_\lambda \in \N[q,t]$: for instance, $q^2{+}t^2\in \N[q,t] $ but $q^2{+}t^2=[3]_{q,t}{-}qt\, [1]_{q,t}$.

Now we have a combinatorial interpretation of $F_{\lambda'}$ using formula \eqref{e-Flambda} of Section \ref{s-intro}.
Firstly,
$$s_{\lambda'}[x_0{+}x_1{+}x_2] = \sum_{ T} x_0^{\kappa_0(T)}x_1^{\kappa_1(T)}x_2^{\kappa_2(T)}, $$
where the sum is over all semi-standard Young tableaux $T$ of shape $\lambda'$ filled with numbers $0,1,2$, and $\kappa_i(T)$ is the number of $i$'s in $T$.
Generic semi-standard Young tableaux $T$ of shape $\lambda'$ are pictured 
in \fref{1}.  For any given tableau $T$, we see that contribution 
to 
$$\frac{s_{\lambda'}[1{+}t{+}t^2]-s_{\lambda'}[1{+}t{+}q]}{t^2{-}q}$$
is 
$$t^{\kappa_1(T)}\frac{t^{\kappa_2(T)} {-} q^{\kappa_2(T)}}{t^2{-}q} = 
t^{\kappa_1(T)}[\kappa_2(T){-}1]_{t^2,q}.$$
Thus 
\begin{align*}
F_{\lambda'} = \sum_{ T} t^{\kappa_1(T)} [\kappa_2(T){-}1]_{t^2,q}.
\end{align*}
This can be interpreted as
\begin{align*}
F_{\lambda'} = \sum_{ T'} t^{\kappa_1(T')} q^{\kappa_2(T')} t^{2\kappa_{\bar 2}(T')},
\end{align*}
where $T'$ ranges over the following objects, which we call enriched tableaux:
$T'$ consists of a semi-standard Young tableau $T$ filled with $0,1,2$ and additional markings on some $2$'s. When reading the $2$'s from left to right in $T$, the corresponding cells in $T'$ contain some undecorated $2$'s (weighted by $q$), followed by some $\bar 2$'s (weighted by $t^2$), followed by a single $\hat 2$ (weighted by $1$). The remaining entries, $0$'s and $1$'s, get weights $1$ and $t$ respectively. See the figures below for examples. For each character $x \in \{0,1,2,\bar 2, \hat 2\}$, $\kappa_{x}(T')$ denotes the number of times $x$ occurs in $T'$.

\begin{theo} \label{EmilysThm}
	For any shape $\lambda$,
	$$
	g_\lambda = \sum_{T'} (tq)^{\kappa_2(T')} [\kappa_1(T'){+}2\kappa_{\bar 2}(T') {-}\kappa_2(T')]_{q,t}
	$$
	where the sum ranges over enriched tableaux $T'$ of shape $\lambda$ 
which
\begin{enumerate}
\item have a $\bar 2$ or $\hat 2$ in the third row,
\item have a $\hat 2$ in the second row \emph{and} fewer than $\kappa_1(T'){+}2\kappa_{\bar 2}(T') {-}\kappa_2(T')$ 1's at the top of columns of height 2, or 
\item have a $\hat 2$ in the bottom row, fewer than $2 \kappa_{\bar 2}$-many $2$'s in the bottom row \emph{and} fewer than $\kappa_1(T'){+}2\kappa_{\bar 2}(T') {-}\kappa_2(T')$ 1's at the top of columns of height 2.
\end{enumerate}
\end{theo}

\begin{proof}
	
	Following the remarks above, for each $i>j$, we will give an injection from enriched tableaux of weight $t^j q^i$ (which are counted by $a_{j,i}$) into those of weight $t^i q^j$ (counted by $a_{i,j}$). The enriched tableaux which are not in the image of this injection will be precisely those enumerated above. 
	
	Let $i>j$ and let $T'$ be an enriched filling of the (french) Young diagram of $\lambda$ with weight $t^j q^i$. Note that $\kappa_{2} > 2\kappa_{\bar 2}$ since undecorated $2$'s are the only entries contributing $q$'s to the weight of $T'$. Note also that $T'$ cannot have a $\bar 2$ or $\hat 2$ in the third row. This is because all $2$'s in the third row are ``balanced" by the $1$'s which must lie beneath them, and the presence of a $\bar 2$ or $\hat 2$ in the third row makes it impossible to gain any more powers of $q$ later in $T'$. Hence we can safely ignore (fix) all columns of height 3.
	
	\textbf{Case 1}: Suppose that the single $\hat 2$ lies in the bottom row. Further suppose that there are at least $2 \kappa_{\bar 2}$-many $2$'s in the bottom row. (Recall that if there are \emph{any} $2$'s in the bottom row, then \emph{all} $\bar 2$'s are also in the bottom row.) Construct $T''$ as follows: Freeze $2 \kappa_{\bar 2}$-many $2$'s in the bottom row. Then exchange the number of $1$'s and unfrozen $2$'s in the bottom row, and also the number of $1$'s and $2$'s above $0$'s (in the second row). Then reorder cells within these rows to make them weakly increasing.

	\begin{figure}[H]
		\begin{center}
			\includegraphics[height=1.2in]{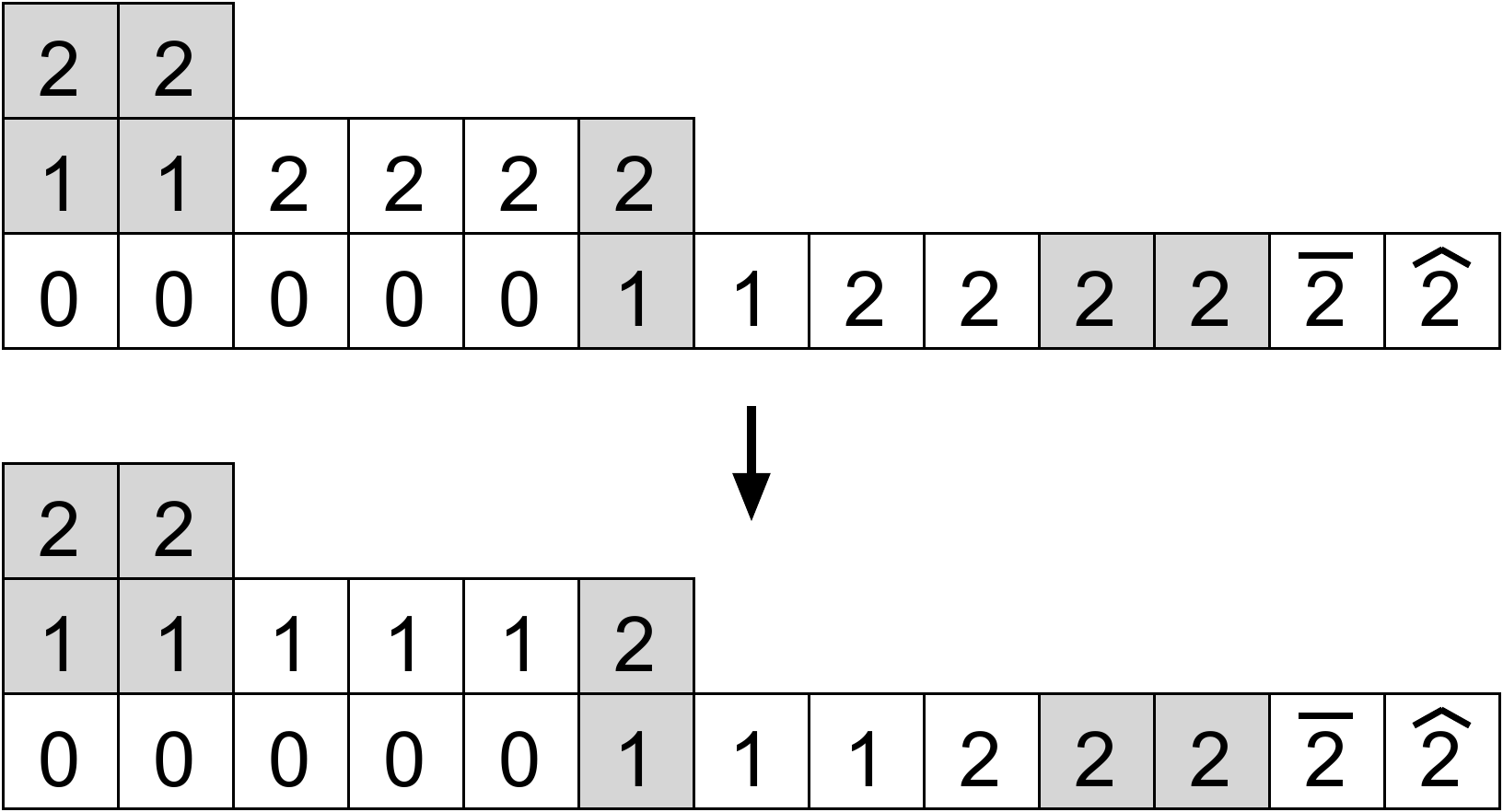}
			\caption{An illustration of Case 1 with $i=10$ and $j=6$. Grey columns are fixed.}
			\label{fig:case1}
		\end{center}
	\end{figure}
	
	\textbf{Case 2:} Suppose that the single $\hat 2$ is not in the bottom row or that there are fewer than $2 \kappa_{\bar 2}$-many $2$'s in the bottom row. Note that in the former situation, there are no $2$'s in the bottom row. Hence, either way, the total weight of all cells in columns of height 1 has a (weakly) larger power of $t$ than $q$. Furthermore, we noted above that the weight of the columns of height 3 has equal powers of $t$ and $q$. Hence the total weight of the columns of height 2 must be $t^b q^a$ for some $a,b$ with $a{-}b \geq i{-}j$. In particular, the number of $2$'s above $0$'s must be at least $i{-}j$ (since $2$'s above $1$'s are ``balanced"). To construct $T''$, simply change the leftmost $i{-}j$ of these to $1$'s (still above $0$'s).
	
	\begin{figure}[H]
		\begin{center}
			\includegraphics[height=1.2in]{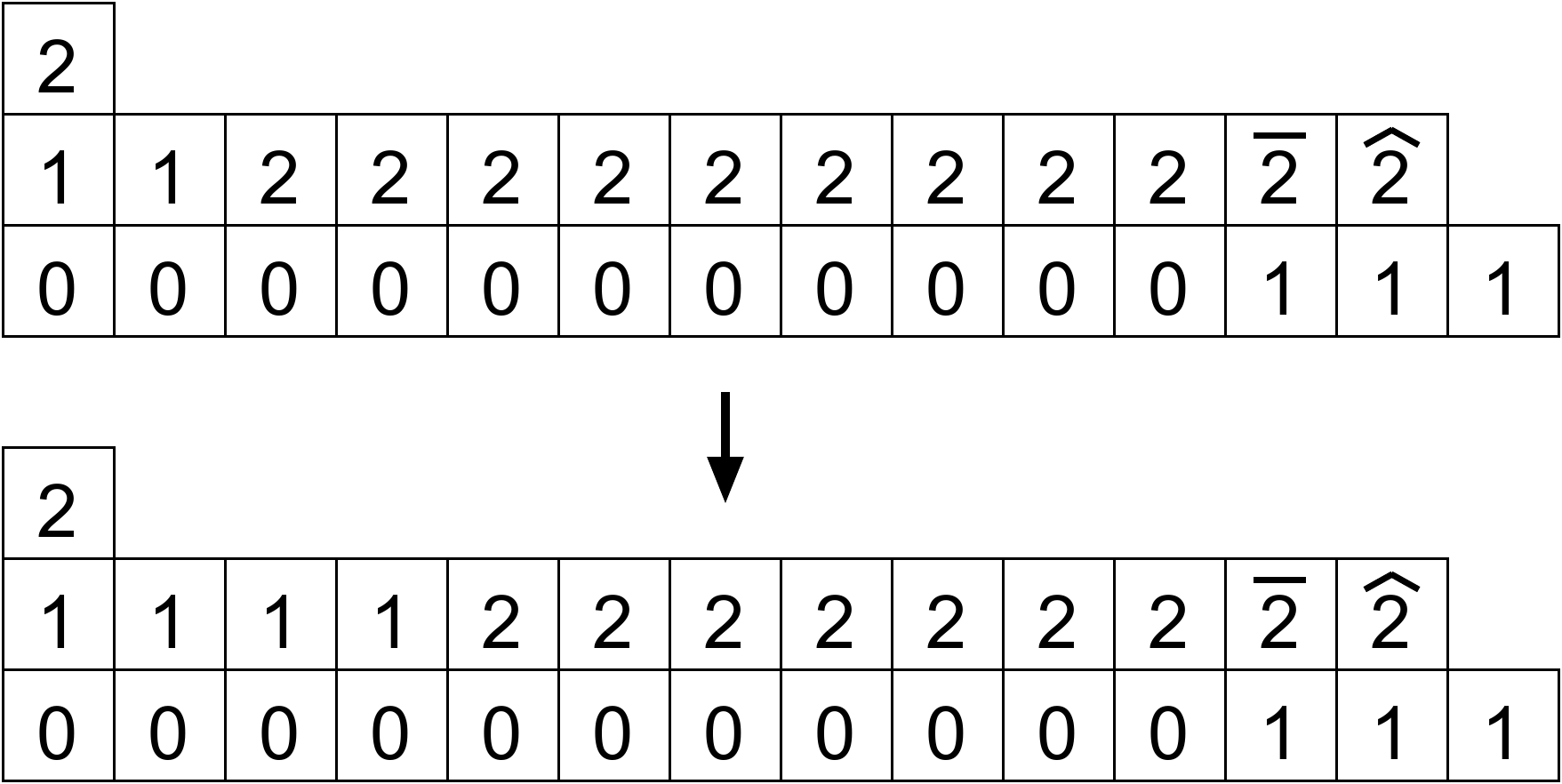}
			\caption{An illustration of Case 2 with $i=10$ and $j=8$.}
			\label{fig:case2}
		\end{center}
	\end{figure}
	
	It is easy to see that each of these maps alone is injective. If you know a particular enriched tableaux $T''$ is the image of a Case 1 tableau $T'$, you can simply freeze $2\kappa_{\bar 2}$-many $2$'s in the bottom row and then switch the roles of 1's and 2's back to reconstruct $T'$. If you know $T''$ is a Case 2 image, you just swap as $(i{-}j)$-many $1$'s for $2$'s at the tops of columns of height $2$.
	
	Furthermore, these images don't intersect: the image of a Case 1 tableau always has the $\hat 2$ and at least $2\kappa_{\bar 2}$-many $2$'s in the bottom row and the image of a Case 2 tableau never does. Hence for any shape $\lambda$ and any $i>j$, these maps together form an injection from enriched tableaux of shape $\lambda$ of weight $t^j q^i$ into those of weight $t^i q^j$.
	
	Using the partial inverses mentioned above, we can see that the enriched tableaux of weight $t^i q^j$ not in the image of our injection are those which
	1) have a $\bar 2$ or $\hat 2$ in the third row;
	2) have a $\hat 2$ in the second row \emph{and} fewer than $i{-}j$ 1's at the top of columns of height 2; or 
	3) have a $\hat 2$ in the bottom row, fewer than $2 \kappa_{\bar 2}$-many $2$'s in the bottom row \emph{and} fewer than $i{-}j$ 1's at the top of columns of height 2. This gives the desired combinatorial interpretation of $ \sum_{i>j} (a_{i,j}{-}a_{j,i}) t^i q^j$ as an enumeration of certain enriched tableaux.
\end{proof}

For example, consider the case $\lambda = (3^1,2^4,1^5)$. Figure \ref{fig:InjEx} shows all enriched tableaux of shape $\lambda'$ and weight $t^5 q^7$ together with their images under the injection above. For the tableaux belonging to Case 1, the frozen cells are shaded. Only the last falls into Case 2 from the proof of the theorem. Then in Figure \ref{fig:LeftoverEx} we give all the remaining enriched tableaux of weight $t^7 q^5$.

Note that in this example there are no such enriched tableaux which have a $\bar 2$ or $\hat 2$ in the third row. This can only happen when all $2$'s are at the tops of columns of height $3$, that is, when the power of $q$ is less than the number of parts of size $3$ in $\lambda$. Similarly, there are no enriched tableaux which have a $\hat 2$ in the second row. This can only happen when the power of $q$ is less than the number of parts of size $2$ or $3$.

\begin{figure}[H]
	\begin{center}
		\includegraphics[height=2.7in]{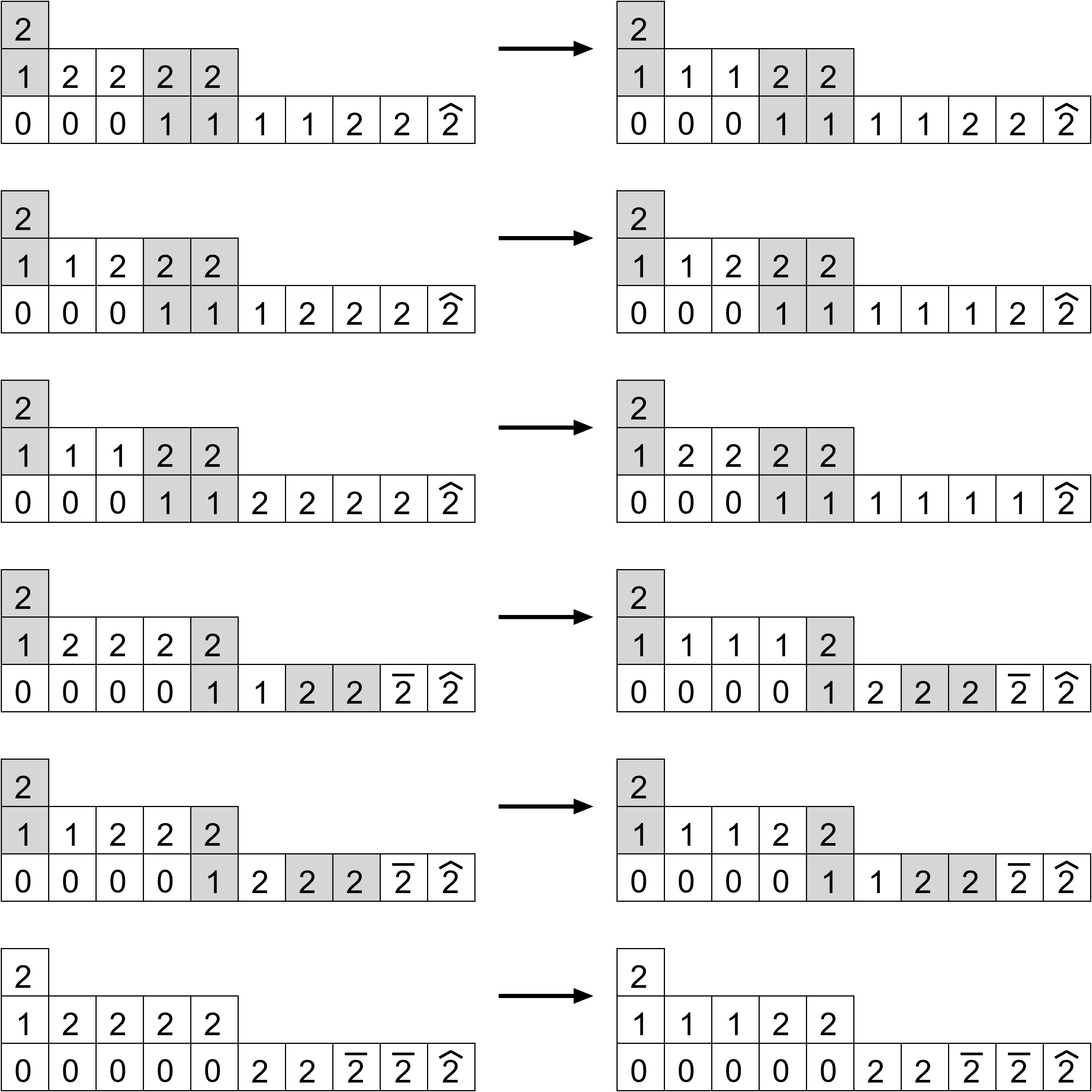}
		\caption{All enriched tableaux of shape $(5,4,1)$ and weight $t^5 q^7$ along with their images under the injection from the proof of Theorem \ref{EmilysThm}.}
		\label{fig:InjEx}
	\end{center}
\end{figure}

\begin{figure}[H]
	\begin{center}
		\includegraphics[height=2.25in]{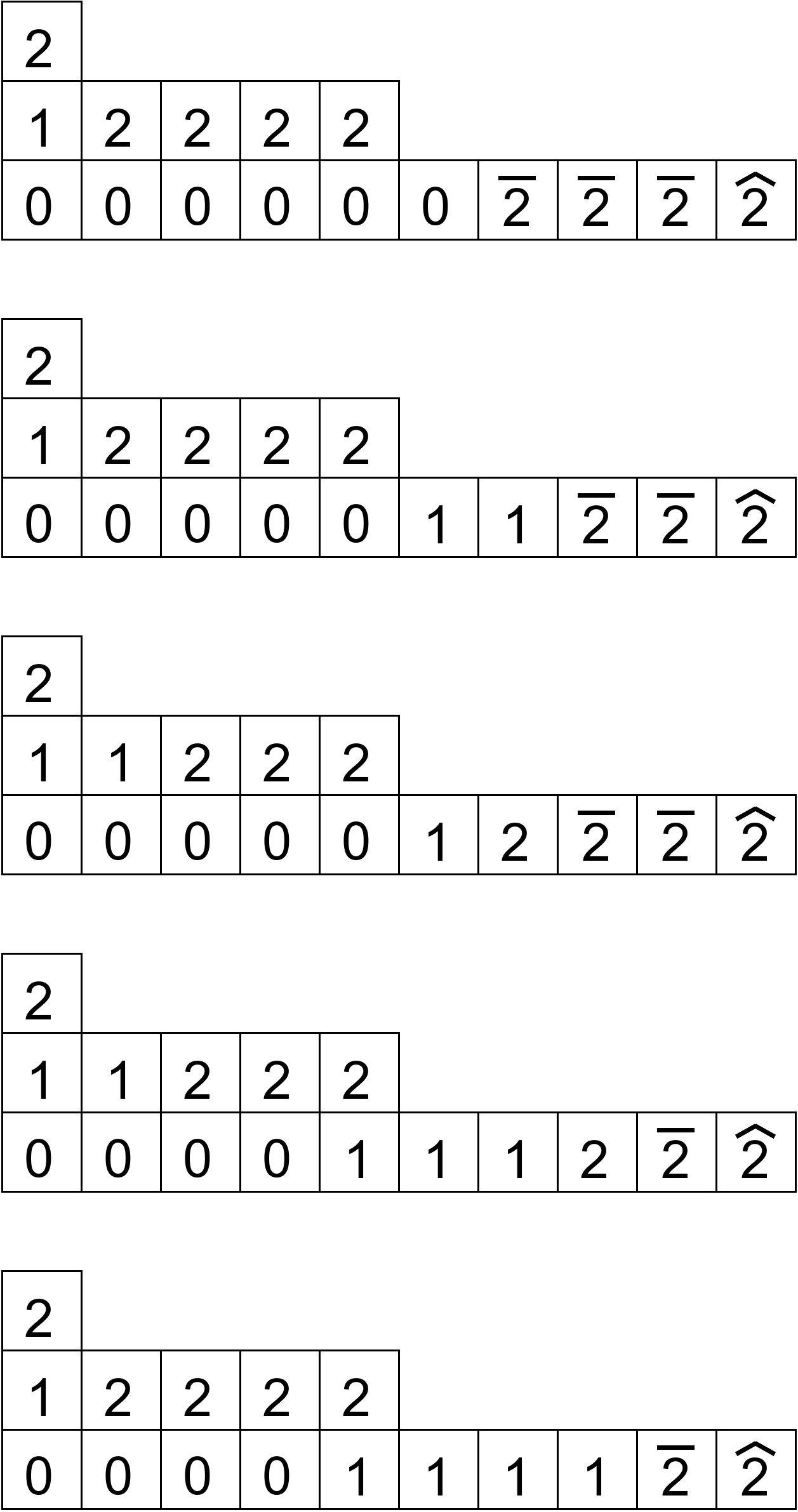}
		\caption{All enriched tableaux of shape $(5,4,1)$ and weight $t^7 q^5$ not included in Figure \ref{fig:InjEx}.}
		\label{fig:LeftoverEx}
	\end{center}
\end{figure}

\section{Proof by Direct Computation\label{s-direct-computation}}

\subsection{Preliminaries}

In this section, we shall show how we can compute an 
explicit formula for $g_{\lambda}$.

We let 
\begin{enumerate}[(1)]
	\item $[n]_q=q^{n-1}{+}q^{n-2}{+}\cdots{+}1=\frac{q^n-1}{q-1}$,
	\item $\qtn{n}=q^{n-1}{+}q^{n-2}t{+}\cdots{+}t^{n-1}=\frac{q^n-t^n}{q-t}$ for $n\geq 0$,
	\item $\qtn{-n}=\frac{q^{-n}-t^{-n}}{q-t}=\frac{-\qtn{n}}{(qt)^n}$ for $n>0$, and 
	\item $\qtn{n\rightarrow m}=\sum_{i=n}^{m}\qtn{i}=\frac{\sum_{i=n}^{m}t^i-\sum_{i=n}^{m}q^i}{t-q}=\frac{t^n[m-n+1]_t-q^n[m-n+1]_q}{t-q}$ or alternatively $\frac{(q-1)(t^{m+1}-t^n)-(t-1)(q^{m+1}-q^n)}{(t-1)(q-1)(t-q)}$.
\end{enumerate}

We know that $g_\lambda =0$ if 
$\lambda'$ has more than 3 rows. Thus we can assume 
that $\lambda'$ has 3 or fewer rows. We let  
$\mathrm{SSYT}(\lambda',012)$ denote the 
set of all semi-standard Young tableaux $T$ of shape $\lambda'$ with cells filled by $\{0,1,2\}$. Given a semi-standard Young tableau $T\in\mathrm{SSYT}(\lambda',012)$, the contribution of $T$ to $g_{\lambda}$ is denoted as $g_T$. This is also known as $T$'s weight. We can write
$$
g_{\lambda}=\sum_{T\in\mathrm{SSYT}(\lambda',012)}g_T.
$$

Since we are only considering the weight of $T\in\mathrm{SSYT}(\lambda',012)$, we can write $T$ in $4$ parts as shown in \fref{1}: $a_1$ -- the part with $3$ rows, $k_1$ -- the part with two rows and the bottom row is filled with $0$'s, $a_2$ -- the part with two rows and the bottom row is filled with $1$'s, $k_2$ -- the part with one row and the fillings are not $0$. If there is no $a_2$ part, there can be a part called $a_0$ at the same place which consists of one row filled with $0$'s. In our weighting scheme for $T\in\mathrm{SSYT}(\lambda',012)$ given below, the weight of any 0 will be 1. 
Hence $a_0$ won't contribute anything to $g_T$ so that we will not consider $a_0$ in our formulas. We define the set $S_{\lambda}[a_1,k_1,a_2,k_2]$ to be the collections of $T$'s having the part composition $[a_1,k_1,a_2,k_2]$. Since $a_1$ and $a_2$ have exactly the same kind of contribution to the formula, we can define
$$
g_{\lambda}[a_1{+}a_2,k_1,k_2]=\sum_{T\in S_{\lambda}[a_1,k_1,a_2,k_2]}g_T.
$$

\begin{figure}[ht]
	\centering
	\begin{tikzpicture}[scale =.5]
	\draw[help lines] (0,0) rectangle (3,3);
	\draw[help lines] (0,0) rectangle (6,2);
	\draw[help lines] (0,0) rectangle (9,2);
	\draw[help lines] (0,0) rectangle (12,1);
	\fillll{1.1}{1}{\cdots}\fillll{2}{1}{0}\fillll{3}{1}{\cdots}\fillll{4}{1}{0}\fillll{5}{1}{\cdots}\fillll{6}{1}{0}
	\fillll{7}{1}{1}\fillll{8}{1}{\cdots}\fillll{9}{1}{1}\fillll{10}{1}{1}\fillll{11}{1}{\cdots}\fillll{12}{1}{2}
	\fillll{1.1}{2}{\cdots}\fillll{2}{2}{1}\fillll{3}{2}{\cdots}\fillll{4}{2}{1}\fillll{5}{2}{\cdots}\fillll{6}{2}{2}
	\fillll{7}{2}{2}\fillll{8}{2}{\cdots}\fillll{9}{2}{2}
	\fillll{1.1}{3}{\cdots}\fillll{2}{3}{2}\fillll{3}{3}{\cdots}
	\draw [thick, blue,decorate,decoration={brace,amplitude=5pt,mirror},xshift=0.4pt,yshift=-0.4pt](0,0) -- (3,0) node[black,midway,yshift=-.4cm] {\footnotesize $a_1$};
	\draw [thick, blue,decorate,decoration={brace,amplitude=5pt,mirror},xshift=0.4pt,yshift=-0.4pt](3,0) -- (6,0) node[black,midway,yshift=-.4cm] {\footnotesize $k_1$};
	\draw [thick, blue,decorate,decoration={brace,amplitude=5pt,mirror},xshift=0.4pt,yshift=-0.4pt](6,0) -- (9,0) node[black,midway,yshift=-.4cm] {\footnotesize $a_2$};
	\draw [thick, blue,decorate,decoration={brace,amplitude=5pt,mirror},xshift=0.4pt,yshift=-0.4pt](9,0) -- (12,0) node[black,midway,yshift=-.4cm] {\footnotesize $k_2$};
	\fillll{6}{-1.3}{\textnormal{(a)}}
	\end{tikzpicture}
	\begin{tikzpicture}[scale =.5]
	\draw (-1.5,1.5) node {or};
	\path (-2.5,1.5);
	\draw[help lines] (0,0) rectangle (3,3);
	\draw[help lines] (0,0) rectangle (6,2);
	\draw[help lines] (0,0) rectangle (9,1);
	\draw[help lines] (0,0) rectangle (12,1);
	\fillll{1.1}{1}{\cdots}\fillll{2}{1}{0}\fillll{3}{1}{\cdots}\fillll{4}{1}{0}\fillll{5}{1}{\cdots}\fillll{6}{1}{0}
	\fillll{7}{1}{0}\fillll{8}{1}{\cdots}\fillll{9}{1}{0}\fillll{10}{1}{1}\fillll{11}{1}{\cdots}\fillll{12}{1}{2}
	\fillll{1.1}{2}{\cdots}\fillll{2}{2}{1}\fillll{3}{2}{\cdots}\fillll{4}{2}{1}\fillll{5}{2}{\cdots}\fillll{6}{2}{2}
	\fillll{1.1}{3}{\cdots}\fillll{2}{3}{2}\fillll{3}{3}{\cdots}
	\draw [thick, blue,decorate,decoration={brace,amplitude=5pt,mirror},xshift=0.4pt,yshift=-0.4pt](0,0) -- (3,0) node[black,midway,yshift=-.4cm] {\footnotesize $a_1$};
	\draw [thick, blue,decorate,decoration={brace,amplitude=5pt,mirror},xshift=0.4pt,yshift=-0.4pt](3,0) -- (6,0) node[black,midway,yshift=-.4cm] {\footnotesize $k_1$};
	\draw [thick, blue,decorate,decoration={brace,amplitude=5pt,mirror},xshift=0.4pt,yshift=-0.4pt](6,0) -- (9,0) node[black,midway,yshift=-.4cm] {\footnotesize $a_0$};
	\draw [thick, blue,decorate,decoration={brace,amplitude=5pt,mirror},xshift=0.4pt,yshift=-0.4pt](9,0) -- (12,0) node[black,midway,yshift=-.4cm] {\footnotesize $k_2$};
	\fillll{6}{-1.3}{\textnormal{(b)}}
	\end{tikzpicture}
	\vspace{-3mm}
	\caption{$T\in\mathrm{SSYT}(\lambda',012)$}
	\label{fig:1}
\end{figure}

By \lref{l3}, we can simplify the formula for $g_{\lambda}$ as:
\begin{eqnarray*}
	g_{\lambda}&=&\langle\Delta_{e_2}e_n[X],s_{\lambda} \rangle \\&=&\frac{id -\tau}{t-q} F_{\lambda'}\\
	&=&\frac{id {-}\tau}{t{-}q} \, \frac{ s_{\lambda'}[1{+}t{+}t^2] - s_{\lambda'}[1{+}t{+}q] }{t^2{-}q}\\
	&=&\frac{s_{\lambda'}[1{+}t{+}t^2]-s_{\lambda'}[1{+}t{+}q]}{(t{-}q)(t^2{-}q)} +\frac{s_{\lambda'}[1{+}q{+}q^2]-s_{\lambda'}[1{+}q{+}t]}{(t{-}q)(t{-}q^2)}.
\end{eqnarray*}

Suppose a Young tableau $T\in\mathrm{SSYT}(\lambda',012)$ has $\omega_1$ $1$'s and $\omega_2$ $2$'s. Then it has weight
\begin{eqnarray*}
	g_T&=&\frac{t^{\omega_1+2\omega_2}{-}t^{\omega_1}q^{\omega_2}}{(t{-}q)(t^2{-}q)}+\frac{q^{\omega_1+2\omega_2}{-}q^{\omega_1}t^{\omega_2}}{(t{-}q)(t{-}q^2)}\\
	&=&\frac{t^{\omega_1}[\omega_2]_{t^2,q}{-}q^{\omega_1}[\omega_2]_{q^2,t}}{t{-}q}.
\end{eqnarray*}

Now we define
$$
w(\omega_1,\omega_2)=\frac{t^{\omega_1}[\omega_2]_{t^2,q}{-}q^{\omega_1}[\omega_2]_{q^2,t}}{t{-}q}
$$
and
$$
W(T)=w(\omega_1,\omega_2)=\frac{t^{\omega_1}[\omega_2]_{t^2,q}{-}q^{\omega_1}[\omega_2]_{q^2,t}}{t{-}q},
$$
Then it is clear that
$$
g_{\lambda}[a_1{+}a_2,k_1,k_2]=\sum_{T\in S_{\lambda}[a_1,k_1,a_2,k_2]}g_T\ =\sum_{T\in S_{\lambda}[a_1,k_1,a_2,k_2]}W(T).
$$
We will use the new weight $W(T)$ to deduce a formula for $g_{\lambda}[a_1{+}a_2,k_1,k_2]$ which will, in turn, allow us to compute an 
explicit formula for $g_{\lambda}$.

\subsection{The computation of $g_{\lambda}[a_1{+}a_2,k_1,k_2]$}

\subsubsection{A formula for  $g_{\lambda}[0,0,k]$}

The set $S[0,0,0,k]$ contains the tableaux $T$ of shape
\raisebox{-11pt}{\begin{tikzpicture}[scale =.3]
	\draw[help lines] (0,0) rectangle (3,1);
	\fillll{1}{1}{1}\fillll{2}{1}{\cdots}\fillll{3}{1}{2}
	\draw [thick, blue,decorate,decoration={brace,amplitude=5pt,mirror},xshift=0.4pt,yshift=-0.4pt](0,0) -- (3,0) node[black,midway,yshift=-.4cm] {\footnotesize $k$};
	\end{tikzpicture}}. If there are $i$ $1$'s, then there will be $k-i$ $2$'s. We then have the following theorem. For any statement $A$, we let 
$\chi(A) =1$ if $A$ is true and $\chi(A) =0$ if $A$ is false. 

\begin{theo}\label{Dun1} $g_{\lambda}[0,0,1] = 0$ and, for 
$k \geq 2$, 
\begin{equation}\label{00k}	
g_{\lambda}[0,0,k]=\sum_{i=0}^{\lfloor (2k-2)/3 \rfloor -\chi(k \ \equiv \ 1 \ \mathrm{mod} \ 3)}(qt)^i \Big[ 
k {-} i {-} \lfloor \textstyle\frac{i+1}{2} \rfloor\rightarrow 2k {-} 2 {-} 3i \Big]_{q,t}.
\end{equation}
\end{theo}

\begin{proof}

It is easy to see by direct calculation that $g_{\lambda}[0,0,1] = 0$. 
Next observe that  for any $r \geq 1$, $\omega(r,0) =0$. Thus 
we need only consider the cases where there is at least one 
2 in the tableau. It follows 
that 

\begin{eqnarray*}
		g_{\lambda}[0,0,k]&=&\sum_{T\in S[0,0,0,k]}g_T\\
		&=&\sum_{i=0}^{k-1}w(i,k{-}i)\\
		&=&\sum_{i=0}^{k-1}\frac{t^{i}[k{-}i]_{t^2,q}{-}q^{i}[k{-}i]_{q^2,t}}{t{-}q}\\
		&=&\sum_{i=0}^{k-1}\frac{\sum_{j=0}^{k-1-i}t^{2k-2j-i-2}q^{j}{-}q^{2k-2j-i-2}t^{j}}{t{-}q}  \\
		&=&\sum_{i=0}^{k-1}\sum_{j=0}^{k-1-i}(qt)^j\qtn{2k{-}3j{-}i{-}2}.
\end{eqnarray*}
Now let $A^{(k)}_{i,j} = (qt)^j\qtn{2k{-}3j{-}i{-}2}$. In \fref{table}, 
we have pictured the array $\{A^{(8)}_{i,j}: 0 {\leq} i {\leq} 7 \ \& \ 
0 {\leq} j {\leq} i\}$. In general, if one 
looks at the first row of the $A^{(k)}_{i,j} = (qt)^j\qtn{2k{-}3j{-}i{-}2}$, 
which is the sequence 
$((qt)^j\qtn{2k{-}3j{-}2})$, the terms will be non-negative 
if $2k{-}2 \geq 3j$, or, equivalently, if $j \leq \lfloor (2k{-}2)/3 \rfloor$. 
We shall show that for any negative terms in the first 
row of the form $(qt)^j[-k]$, the first $k{+}1$ terms along the anti-diagonal 
starting at that position will sum to 0.  This will leave us 
only with positive terms corresponding to sum stated in the theorem. 
For example, in \fref{table}, one can easily compute 
that the sum of the first two terms of the anti-diagonal starting 
at the term $(qt)^5[-1]_{q,t}$ equals 0, the sum 
of the first five terms of the anti-diagonal starting 
at the term $(qt)^6[-4]_{q,t}$ equals 0, and the sum of 
the first eight terms of the anti-diagonal starting at the 
term $(qt)^7[-7]_{q,t}$ equals 0. These are the terms corresponding to the 
green, blue, and red diagonals respectively. In this case, 
we see that $g_\lambda[0,0,8]$ equals 
$$[8 \rightarrow 14]_{q,t}+ qt[6 \rightarrow 11]_{q,t} +
(qt)^2[5 \rightarrow 8]+(qt)^3[3 \rightarrow 5] + 
(qt)^4[2 \rightarrow 2] 
$$ 
which are exactly the terms predicted by the theorem.

\begin{figure}[H]
	\begin{center}
		\vskip -3mm
		\scalebox{1.5}{\includegraphics[width=3.0in]{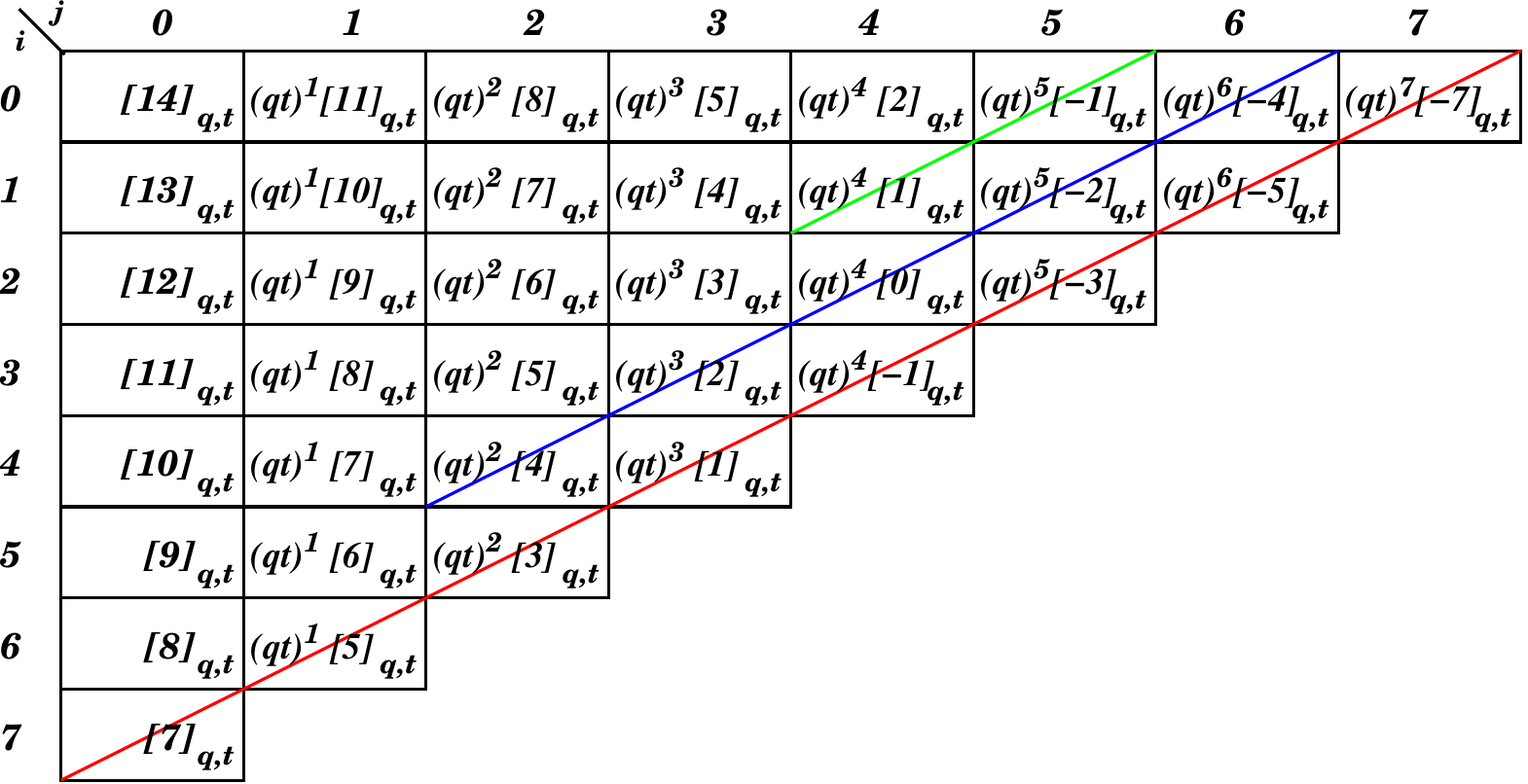}}
		\caption{The table of $A^{(8)}_{i,j}$.}
		\label{fig:table}
	\end{center}
\end{figure}

The proof requires a careful case by case analysis by 
considering the parity of $k$ modulo 3. 
Note that  
\begin{enumerate}
\item if $k =3t$, then $\lfloor (2k{-}2)/3 \rfloor =2t{-}1$, 
\item if $k =3t{+}1$, then $\lfloor (2k{-}2)/3 \rfloor =2t$, and 
\item if $k =3t{+}2$, then $\lfloor (2k{-}2)/3 \rfloor =2t$.
\end{enumerate}

\ \\
{\bf Case 1.}  $k =3t$.  \\

The negative terms in the first row are 
$$(qt)^{2t-1+s}[6t{-}2{-}3(2t{-}1{+}s)]_{q,t} = 
(qt)^{2t-1+s}[-3s{+}1]_{q,t}$$ 
for $s =1, \ldots, t$. In particular, the last term 
in the first row equals $(qt)^{3t-1}[-3t{+}1]$ and the first 
negative term is $A^{(3t)}_{0,2t} = (qt)^{2t}[-2]$.

Then we have two subcases depending 
on whether $s$ is even or odd. \\
\ \\
{\bf Subcase 1.1.} $s = 2r$. \\
\ \\
In this case, $A^{(3t)}_{0,2t-1+2r} = q^{2t+2r-1}[-6r{+}1]_{q,t}$. We claim that 
$\sum_{a=0}^{6r-1} A^{(3t)}_{a,2t-1+2r-a} =0$. 
We shall prove this by showing that for all $0 \leq a \leq 3r{-}1$, 
$$A^{(3t)}_{a,2t-1+2r-a} =- A^{(3t)}_{6r-1-a,2t-1+2r-(6r-1-a)} = 
-A^{(3t)}_{6r-1-a,2t-4r+a}.$$
Note that 
\begin{eqnarray*}
A^{(3t)}_{a,2t-1+2r-a} &=&(qt)^{2t-1+2r-a}[6t{-}2{-}a{-}3(2t{-}1{+}2r{-}a)]_{q,t}\\
&=&  (qt)^{2t-1+2r-a}[-6r{+}1{+}2a]_{q,t} \\
&=& - (qt)^{2t-1+2r-a-(6r-1-2a)}[6r{-}1{+}2a]_{q,t} \\
&=& - (qt)^{2t-4r+a}[6r{-}1{+}2a]_{q,t}.
\end{eqnarray*}
On the other hand, 
\begin{multline*}A^{(3t)}_{6r{-}1{-}a,2t{-}4r{+}a} = \\
(qt)^{2t-4r+a}[6t{-}2{-}(6r{-}1{-}a) {-}
3(2t{-}4r{+}a)]_{q,t}=
(qt)^{2t-4r+a}[6t{-}1{+}2a]_{q,t}
\end{multline*}
as desired.
\ \\

{\bf Subcase 1.2.} $s = 2r{+}1$.\\
\ \\
In this case, $A^{(3t)}_{0,2t-1+2r+1} = q^{2t+2r-1}[-6r{-}2]_{q,t}$. 
We  claim that 
$\sum_{a=0}^{6r+2} A^{(3t)}_{a,2t+2r-a} =0$. 
First note that 
\begin{multline*} A^{(3t)}_{3r+1,2t+2r-(3r+1)} = A^{(3t)}_{3r+1,2t-r-1} = \\
(qt)^{2t-r-1}[6t{-}2{-}(3r{+}1){-}3(2t{-}r{-}1)]_{q,t} = (qt)^{2t-r-1}[0]_{q,t} =0.
\end{multline*}
Thus we can prove our claim if we show  that $0 \leq a \leq 3r$, 
$$A^{(3t)}_{a,2t+2r-a} =- A^{(3t)}_{6r+2-a,2t+2r-(6r+2-a)} = 
-A^{(3t)}_{6r+2-a,2t-4r-2+a}.$$
Note that 
\begin{eqnarray*}
A^{(3t)}_{a,2t+2r-a} &=&(qt)^{2t+2r-a}[6t{-}2{-}a{-}3(2t{+}2r{-}a)]_{q,t}\\
&=&  (qt)^{2t+2r-a}[-6r{-}2{+}2a]_{q,t} \\
&=& - (qt)^{2t+2r-a-(6r+2-2a)}[6r{+}2{-}2a]_{q,t} \\
&=& - (qt)^{2t-4r-2+a}[6r{+}2{-}2a]_{q,t}.
\end{eqnarray*}
On the other hand, 
\begin{align*} A^{(3t)}_{6r+2-a,2t-4r-2+a} &= (qt)^{2t-4r-2+a}[6t{-}2{-}(6r{+}2{-}a) {-}
3(2t{-}4r{-}2{+}a)]_{q,t}\\
&=(qt)^{2t{-}4r{+}a}[6r{+}2{-}2a]_{q,t}.
\end{align*}

Observe that the bottom term of the $r$-th column of the array 
$\{A^{(3t)}_{i,j}\}_{i=0, \ldots, 3t{-}1 \, \& \,0 \leq j \leq i}$ 
is $A_{3t-1-r,r}$. Our computations above show that 
in the array $\{A^{(3t)}_{i,j}\}_{i=0, \ldots, 3t{-}1 \, \& \, 0 \leq j \leq i}$, 
 the first $3s$ terms of any anti-diagonal starting at $A_{0,2t-1+s}$ sum to 0 for $s =1, \ldots, t$. 
This means 
that the corresponding terms in the array make no contribution to 
$g_{\lambda}[0,0,k]$. It follows that we can ignore all the terms 
in columns $2t, \ldots, 3t{-}1$.  Note that the first $3t$ terms 
of the anti-diagonal starting at $A^{(3t)}_{0,3t-1}$ cancel out the bottom 
term in each column. Next  the first $3t-3$ terms 
of the anti-diagonal starting at $A^{(3t)}_{0,3t-2}$ reach only to column 2 
so they will cancel out the next to last term in columns  
$2, \ldots, 2t{-}1$. Then  the first $3t{-}6$ terms 
of the anti-diagonal starting at $A^{(3t)}_{0,3t-3}$ reach only to column 4 
so they will cancel out the second to last terms in columns  
$4, \ldots, 2t{-}1$.  Continuing on in this way, we finally see 
that the 3 anti-diagonal terms starting at $A^{(3t)}_{0,2t}$ will only cancel 
out terms in columns $2t{-}2$ and $2t{-}1$.  It follows that 
for $r =0, \ldots, t{-}1$, we can ignore that last $r{+}1$ terms 
in columns $2r$ and $2r{+}1$. 
This means that if $0 \leq r \leq t{-}1$, the lowest  term 
that can contribute to  
$g_{\lambda}[0,0,k]$ in column $2r$ is  
\begin{eqnarray*}
A^{(3t)}_{3t-1 -2r -(r+1),2r} &=& A_{3t-3r-2,2r}= (qt)^{2r}[6t{-}2{-}(3t{-}3r{-}2){-}3(2r)]_{q,t} \\
&=& (qt)^{2r}[3t{-}3r]_{q,t} = [3t{-}(2r) {-}\lfloor 2r{+}1/2 \rfloor]_{q,t}.
\end{eqnarray*}
Note that the top element in column $2r$ is $A^{(3t)}_{0,2r} = 
(qt)^{2r}[3t{-}2{-}3(2r)]_{q,t}$. 
Since the $q,t$-numbers of the terms in column $2r$ increase by 1 as 
one moves up, it follows that the contribution of 
column $2r$  to $g_{\lambda}[0,0,k]$ is 
$(qt)^{2r}[k{-}(2r) {-}\lfloor 2r{+}1/2\rfloor \rightarrow 2k{-}2{-}3(2r)]_{q,t}$ as predicted by 
our formula. 

Similarly, if $0 \leq r \leq t-1$, the lowest  term 
that can contribute to  
$g_{\lambda}[0,0,k]$ in column $2r{+}1$ is  
\begin{eqnarray*}
A^{(3t)}_{3t-1 -(2r+1) -(r+1),2r+1} &=& A_{3t-3r-3,2r+1}= (qt)^{2r+1}[6t{-}2{-}(3t{-}3r{-}3){-}3(2r{+}1)]_{q,t} \\
&=& (qt)^{2r+1}[3t{-}3r{-}2]_{q,t} = [3t{-}(2r{+}1) {-}\lfloor 2r{+}2/2\rfloor]_{q,t}.
\end{eqnarray*}
Note that the top element in column $2r{+}1$ is 
$A^{(3t)}_{0,2r+1}= (qt)^{2r+1}[3t{-}2{-}3(2r{+}1)]_{q,t}$. 
Since the $q,t$-numbers in the terms in column $2r{+}1$ increase by 1 as 
one moves up, it follows that the contribution of 
column $2r{+}1$  to $g_{\lambda}[0,0,k]$ is 
$(qt)^{2r+1}[k{-}(2r{+}1) {-}\lfloor 2r{+}1/2 \rfloor \rightarrow 2k{-}2{-}3(2r{+}1)]_{q,t}$ as predicted by 
our formula.

Thus our formula holds in this case. \\
\ \\
{\bf Case 2.}  $k =3t{+}1$.  \\

The negative terms in the first row are 
$$(qt)^{2t+s}[6t{+}2{-}2{-}3(2t{+}s)]_{q,t} = 
(qt)^{2t+s}[-3s]_{q,t}$$ 
for $s =1, \ldots, t$. In particular, the last term 
in the first row equals $(qt)^{3t}[-3t]$ and the first 
negative term is $A^{(3t+1)}_{0,2t+1} = (qt)^{2t+1}[-3]$.

Then as in Case 1, we have two subcases depending 
on whether $s$ is even or odd. \\
\ \\
{\bf Subcase 2.1.} $s = 2r$. \\
\ \\
In this case, $A^{(3t+1)}_{0,2t+2r} = q^{2t+2r}[-6r]_{q,t}$. We claim that 
$\sum_{a=0}^{6r} A^{(3t)}_{a,2t-1+2r-a} =0$. 
First observe that 
$$A^{(3t+1)}_{3r,2t+2r-(3r)} = q^{2t-r}[6t{+}2{-}2{-}3r{-}3(2t{-}r)]_{q,t}  =q^{2t-r}
[0]_{q,t}.$$
Thus we can prove our claim by showing that for  $0 \leq a \leq 3r-1$, 
$$A^{(3t+1)}_{a,2t+2r-a} =- A^{(3t+1)}_{6r-a,2t+2r-(6r-a)}.$$
This is a straightforward calculation so we will not include the details here. \\
\ \\
{\bf Subcase 2.2.} $s = 2r{+}1$.\\
\ \\
In this case, $A^{(3t+1)}_{0,2t+2r+1} = q^{2t+2r+1}[-6r{-}3]_{q,t}$. 
We  claim that 
$\sum_{a=0}^{6r+3} A^{(3t+1)}_{a,2t+2r+1-a} =0$. 
In this case, one can easily check  that $0 \leq a \leq 3r{+}1$, 
$$A^{(3t+1)}_{a,2t+2r+1-a} =- A^{(3t+1)}_{6r+3-a,2t+2r+1-(6r+3-a)}$$
so we shall not include the details here. \\
\ \\

Next observe that the bottom term of the array 
$\{A^{(3t+1)}_{i,j}\}_{i=0, \ldots, 3t \ \& \ 0 \leq j \leq i}$
in the $r$-th column 
is $A_{3t-r,r}$. Our computations above show that 
in the array $\{A^{(3t+1)}_{i,j}\}_{i=0, \ldots, 3t \ \& \ 0 \leq j \leq i}$, 
 the first $3s{+}1$ terms of any anti-diagonal 
terms starting at $A^{(3t+1)}_{0,2t+s}$ sum to 0 for $s =1, \ldots, t$. 
This means 
that the corresponding terms in the array make no contribution to 
$g_{\lambda}[0,0,k]$. It follows that we can ignore all the terms 
in columns $2t{+}1, \ldots, 3t$.   One can use a similar reasoning 
as we used in Case 1 to show that for $r =0, \ldots, t{-}1$, we can ignore the 
bottom $r{+}1$ terms 
in columns $2r$ and $2r{+}1$. Moreover, we can ignore the bottom $t$ terms in 
column $2t$.   This is because $A^{(3t+1)}_{0,2t+1} = [-3]$, which 
means that the first four terms of the anti-diagonal starting at 
$A^{(3t+1)}_{0,2t+1}$ will cancel terms in columns $2t{-}2$, $2t{-}1$, and 
$2t$. It follows that if $0 \leq r \leq t{-}1$, the lowest  term 
that can contribute to  
$g_{\lambda}[0,0,k]$ in column $2r$ is  
\begin{eqnarray*}
A^{(3t+1)}_{3t -2r -(r+1),2r} &=& A^{(3t+1)}_{3t-3r-1,2r}= (qt)^{2r}[6t{+}2{-}2{-}(3t{-}3r{-}1){-}3(2r)]_{q,t} \\
&=& (qt)^{2r}[3t{-}3r{+}1]_{q,t} = [3t{+}1{-}(2r) {-}\lfloor 2r{+}1/2 \rfloor]_{q,t}.
\end{eqnarray*}
Note that the top element in column $2r$ is $A^{(3t+1)}_{0,2r} = 
(qt)^{2r}[2(3t{+}1){-}2{-}3(2r)]_{q,t}$. 
Since the $q,t$-numbers in the terms in column $2r$ increase by 1 as 
one moves up, it follows that the contribution of 
column $2r$  to $g_{\lambda}[0,0,k]$ is 
$(qt)^{2r}[k{-}(2r) {-}\lfloor 2r{+}1/2\rfloor \rightarrow 2k{-}2{-}3(2r)]_{q,t}$ as predicted by 
our formula.

Similarly, if $0 \leq r \leq t{-}1$, the lowest  term 
that can contribute to  
$g_{\lambda}[0,0,k]$ in column $2r+1$ is  
\begin{eqnarray*}
A^{(3t+1)}_{3t -(2r+1) -(r+1),2r+1} &=& A^{(3t+1)}_{3t-3r-2,2r+1}= 
(qt)^{2r+1}[6t{+}2{-}2{-}(3t{-}3r{-}2){-}3(2r{+}1)]_{q,t} \\
&=& (qt)^{2r+1}[3t{-}3r{-}1]_{q,t} = [(3t{+}1){-}(2r{+}1) {-}\lfloor 2r{+}2/2\rfloor]_{q,t}.
\end{eqnarray*}
Note that the top element in column $2r{+}1$ is 
$A^{(3t+1)}_{0,2r+1}= (qt)^{2r+1}[2k{-}2{-}3(2r{+}1)]_{q,t}$. 
Since the $q,t$-numbers in the terms in column $2r{+}1$ increase by 1 as 
one moves up, it follows that the contribution of 
column $2r{+}1$  to $g_{\lambda}[0,0,k]$ is 
$(qt)^{2r+1}[k{-}(2r{+}1) {-}\lfloor 2r{+}1/2 \rfloor \rightarrow 2k{-}2{-}3(2r{+}1)]_{q,t}$ as predicted by 
our formula.

Finally in column $2t$, the lowest term that can 
contribute to $g_{\lambda}[0,0,k]$ is  
\begin{eqnarray*}
A^{(3t+1)}_{3t -(2t) -(t),2r+1} &=& A^{(3t+1)}_{0,2t}= 
(qt)^{2t}[6t{+}2{-}2{-}3(2t)]_{q,t} \\
&=& (qt)^{2t}[0]_{q,t}.
\end{eqnarray*}
Thus this column makes no contribution which is why 
we exclude this term from the sum. Note that 
in this case $3t{+}1{-}2t {-}\lfloor 2t{+}1 \rfloor =1$ while 
$2k{-}2{-}3(2t) = 6t{+}2{-}2{-}6t=0$ so that 
$[k{-}2t{-}\lfloor 2t{+}1 \rfloor  \rightarrow 2k{-}2{-}3(6t)] =[1 \rightarrow 0]$ 
which is an empty sum. 

Thus our formula holds in this case. \\
\ \\
{\bf Case 3.}  $k =3t{+}2$.  \\

Then the negative terms in the first row are 
$$(qt)^{2t+s}[6t{+}4{-}2{-}3(2t{+}s)]_{q,t} = 
(qt)^{2t+s}[-3s{+}2]_{q,t}$$ 
for $s =1, \ldots, t{+}1$. In particular, the last term 
in the first row equals $(qt)^{3t+1}[-3t{-}1]$ and the first 
negative term is $A^{(3t+2)}_{0,2t+1} = (qt)^{2t+1}[-1]$.

Then as before, we have two subcases depending 
on whether $s$ is even or odd. \\
\ \\
{\bf Subcase 3.1.} $s = 2r$. \\
\ \\
In this case, $A^{(3t+2)}_{0,2t+2r} = q^{2t+2r}[-6r{+}2]_{q,t}$. We claim that 
$\sum_{a=0}^{6r-2} A^{(3t)}_{a,2t-1+2r-a} =0$. 
First observe that 
$$A^{(3t+2)}_{3r-1,2t+2r-(3r-1)} = q^{2t-r+1}[6t{+}4{-}2{-}(3r{-}1){-}3(2t{-}r{+}1)]_{q,t}  
=q^{2t-r+1}
[0]_{q,t}.$$
Thus we can prove our claim by showing that for  $0 \leq a \leq 3r{-}2$, 
$$A^{(3t+2)}_{a,2t+2r-a} =- A^{(3t+2)}_{6r-2-a,2t+2r-(6r-2-a)}.$$
This is a straightforward calculation so we will not include the details here. \\
\ \\
{\bf Subcase 3.2.} $s = 2r{+}1$.\\
\ \\
In this case, $A^{(3t+2)}_{0,2t+2r+1} = q^{2t+2r+1}[-6r{-}1]_{q,t}$. 
We  claim that 
$\sum_{a=0}^{6r+1} A^{(3t+2)}_{a,2t+2r+1-a} =0$. 
In this case, one can easily check that for $0 \leq a \leq 3r$, 
$$A^{(3t+1)}_{a,2t+2r+1-a} =- A^{(3t+1)}_{6r+1-a,2t+2r+1-(6r+1-a)}$$
so we shall not include the details here. \\
\ \\

Next we observe that the bottom term of the array 
$\{A^{(3t+2)}_{i,j}\}_{i=0, \ldots, 3t \ \& \ 0 \leq j \leq i}$
in the $r$-th column 
is $A_{3t+1-r,r}$. Our computations above have shown that 
in the array $\{A^{(3t+2)}_{i,j}\}_{i=0, \ldots, 3t+1 \ \& \ 0 \leq j \leq i}$, 
 the first $3s-1$ terms of any anti-diagonal starting at $A^{(3t+2)}_{0,2t+s}$ sum to 0 for $s =1, \ldots, t{+}1$. 
This means 
that the corresponding terms in the array make no contribution to 
$g_{\lambda}[0,0,k]$. It follows that we can ignore all the terms 
in columns $2t{+}1, \ldots, 3t{+}1$.   One can use a similar reasoning 
as we used in Case 1 to show that for $r =0, \ldots, t{-}1$, we can ignore the 
bottom $r{+}1$ terms 
in columns $2r$ and $2r{+}1$. We can also ignore the bottom $t{+}1$ terms in 
column $2t$.   This is because $A^{(3t+2)}_{0,2t+1} = (qt)^{2t+1}[-1]$ 
so that the sum of first two anti-diagonal terms starting at  
$A^{(3t+2)}_{0,2t+1}$ will only cancel elements in columns $2t$ and $2t{+}1$.

This means that if $0 \leq r \leq t{-}1$, the lowest  term 
that can contribute to  
$g_{\lambda}[0,0,k]$ in column $2r$ is  
\begin{eqnarray*}
A^{(3t+2)}_{3t+1 -2r -(r+1),2r} &=& A^{(3t+2)}_{3t-3r,2r}= (qt)^{2r}[6t{+}4{-}2{-}(3t{-}3r){-}3(2r)]_{q,t} \\
&=& (qt)^{2r}[3t{-}3r{+}2]_{q,t} = [3t{+}2{-}(2r) {-}\lfloor 2r{+}1/2 \rfloor]_{q,t}.
\end{eqnarray*}
Note that the top element in column $2r$ is $A^{(3t+2)}_{0,2r} = 
(qt)^{2r}[2(3t{+}2){-}2{-}3(2r)]_{q,t}$. 
Since the $q,t$-numbers in the terms in column $2r$ increase by 1 as 
one moves up, it follows that the contribution of 
column $2r$  to $g_{\lambda}[0,0,k]$ is 
$(qt)^{2r}[k{-}(2r) {-}\lfloor 2r{+}1/2\rfloor \rightarrow 2k{-}2{-}3(2r)]_{q,t}$ as predicted by 
our formula.

Similarly, if $0 \leq r \leq t{-}1$, the lowest  term 
that can contribute to  
$g_{\lambda}[0,0,k]$ in column $2r{+}1$ is  
\begin{eqnarray*}
A^{(3t+2)}_{3t+1 -(2r+1) -(r+1),2r+1} &=& A^{(3t+2)}_{3t-3r-1,2r+1}= 
(qt)^{2r+1}[6t{+}4{-}2{-}(3t{-}3r{-}1){-}3(2r{+}1)]_{q,t} \\
&=& (qt)^{2r+1}[3t{-}3r]_{q,t} = [(3t{+}2){-}(2r{+}1) {-}\lfloor 2r{+}2/2\rfloor]_{q,t}.
\end{eqnarray*}
Note that the top element in column $2r{+}1$ is 
$A^{(3t+2)}_{0,2r+1}= (qt)^{2r+1}[2(3t{+}2){-}2{-}3(2r{+}1)]_{q,t}$. 
Since the $q,t$-numbers in the terms in column $2r{+}1$ increase by 1 as 
one moves up, it follows that the contribution of 
column $2r{+}1$  to $g_{\lambda}[0,0,k]$ is 
$(qt)^{2r+1}[k{-}(2r{+}1) {-}\lfloor 2r{+}1/2 \rfloor \rightarrow 2k{-}2{-}3(2r{+}1)]_{q,t}$ as predicted by 
our formula.

Finally for column $2t$, the lowest term that can 
contribute to $g_{\lambda}[0,0,k]$ in column $2t$ is  
\begin{eqnarray*}
A^{(3t+2)}_{3t+1 -(2t) -(t+1),2t} &=& A^{(3t+2)}_{0,2t}= 
(qt)^{2t}[6t{+}4{-}2{-}3(2t)]_{q,t} \\
&=& (qt)^{2t}[2]_{q,t} = [(3t{+}2){-}(2t) {-}\lfloor 2t{+}1/2\rfloor]_{q,t}.
\end{eqnarray*}
It follows that the contribution of 
column $2t$  to $g_{\lambda}[0,0,k]$ is 
$(qt)^{2t}[k{-}(2t) {-}\lfloor 2t{+}1/2 \rfloor \rightarrow 2k{-}2{-}3(2t)]_{q,t} =
(qt)^{2t}[2]$ as predicted by 
our formula.

Thus our formula holds in this case which completes our proof.
\end{proof}

For example, we have
	\begin{eqnarray*}
		g_{\lambda}[0,0,12]&=&\sum_{i=0}^{7}(qt)^i\Big[ \textstyle
12 {-}i {-} \lfloor \frac{i{+}1}{2} \rfloor \rightarrow 22{-}3i \Big]_{q,t}\\
		&=&\qtn{12\rightarrow 22}+(qt)\qtn{10\rightarrow 19}+(qt)^2\qtn{9\rightarrow 16}+(qt)^3\qtn{7\rightarrow 13}\\
		& &+(qt)^4\qtn{6\rightarrow 10}+(qt)^5\qtn{4\rightarrow 7}+(qt)^6\qtn{3\rightarrow 4}+(qt)^7\qtn{1}.
	\end{eqnarray*}
	
	\subsubsection{A formula for $g_{\lambda}[a,0,k]$}
	
	\begin{theo}\label{Dun2}
		$$
		g_{\lambda}[a,0,k]=(qt)^ag_{\lambda}[0,0,k]+\sum_{i=1}^{a}(qt)^{a-i}\qtn{k{+}3i\rightarrow 2k{+}3i}.
		$$
	\end{theo}
	
	\begin{proof}
	We have
		\begin{eqnarray*}
			g_{\lambda}[a,0,k]&=&\sum_{T\in S[a,0,0,k]}g_T\\
			&=&\sum_{i=0}^{k}w(a{+}i,a{+}k{-}i)\\
			&=&\sum_{i=0}^{k}\frac{t^{a+i}[a{+}k{-}i]_{t^2,q}-q^{a+i}[a{+}k{-}i]_{q^2,t}}{t{-}q}.\\
		\end{eqnarray*}
		
		Notice that	
		$$
		[a{+}k{-}i]_{t^2,q}=q^a[k{-}i]_{t^2,q}+\sum_{j=0}^{a-1}t^{2(k{-}i{+}a{-}j{-}1)}q^j
		\vspace{-3mm}
		$$
		and
		$$
		\vspace{-3mm}
		[a+k-i]_{q^2,t}=t^a[k{-}i]_{q^2,t}+\sum_{j=0}^{a-1}q^{2(k{-}i{+}a{-}j{-}1)}t^j.
		$$	
		By plugging in these new equations we can get
		\begin{eqnarray*}
			g_{\lambda}[a,0,k]&=&t^aq^a\sum_{i=0}^{k}\frac{t^{i}[k{-}i]_{t^2,q}-q^{i}[k{-}i]_{q^2,t}}{t{-}q}\\
			& &+\sum_{i=0}^k \sum_{j=0}^{a-1}(qt)^j \frac{t^{2k-i+3a-3j-2}{-}q^{2k-i+3a-3j-2}}{t{-}q}\\
			&=&(qt)^ag_{\lambda}[0,0,k]+\sum_{j=0}^{a-1}(qt)^j\sum_{i=0}^{k}\qtn{2k{-}i{+}3a{-}3j{-}2}\\
			&=&(qt)^ag_{\lambda}[0,0,k]+\sum_{i=0}^{a-1}(qt)^i\qtn{k{+}3a{-}3i{-}2\rightarrow 2k{+}3a{-}3i{-}2}\\
			&=&(qt)^ag_{\lambda}[0,0,k]+\sum_{i=1}^{a}(qt)^{a-i}\qtn{k{+}3i{-}2\rightarrow 2k{+}3i{-}2}\ .\qedhere
		\end{eqnarray*}\end{proof}
		
		\subsubsection{ The computation of $g_{\lambda}[a,k_1,k_2]$}
		
		We still need to add $k_1$ to complete the formula. Note that the function $g_{\lambda}[a,k_1,k_2]=g_{\lambda}[a,k_2,k_1]$. So without loss of generality, we suppose $k_1\leq k_2$.
		\begin{theo}\label{Dun3}
			For $k_1\leq k_2$, we have
			$$
			g_{\lambda}[a,k_1,k_2]=\sum_{i=0}^{k_1}g_{\lambda}[a{+}i,0,k_1{+}k_2{-}2i].
			$$
		\end{theo}
		
		\begin{proof}
			\begin{eqnarray*}
				g_{\lambda}[a,k_1,k_2]&=&\sum_{T\in S[a,k_1,0,k_2]}g_T\\
				&=&\sum_{j=0}^{k_1}\sum_{i=0}^{k_2}w(a{+}i{+}j,a{+}k_1{+}k_2{-}i{-}j)\\
				&=&\sum_{i=0}^{k_1}\sum_{j=0}^{k_1+k_2-2i}w(a{+}i{+}j,a{+}k_1{+}k_2{-}i{-}j)\\
				&=&\sum_{i=0}^{k_1}g_{\lambda}[a+i,0,k_1+k_2-2i].\qedhere
			\end{eqnarray*}
		\end{proof}
		
		\subsection{Formula for $g_{\lambda}$}
		For any $\lambda=(3^a 2^b1^c)$, $\lambda'$ has the shape
		\raisebox{-11pt}{\begin{tikzpicture}[scale =.2]
			\draw[help lines] (0,0) rectangle (3,3);
			\draw[help lines] (0,0) rectangle (6,2);
			\draw[help lines] (0,0) rectangle (9,1);
			
			\draw [thick, blue,decorate,decoration={brace,amplitude=5pt,mirror},xshift=0.4pt,yshift=-0.4pt](0,0) -- (3,0) node[black,midway,yshift=-.4cm] {\footnotesize $a$};
			\draw [thick, blue,decorate,decoration={brace,amplitude=5pt,mirror},xshift=0.4pt,yshift=-0.4pt](3,0) -- (6,0) node[black,midway,yshift=-.4cm] {\footnotesize $b$};
			\draw [thick, blue,decorate,decoration={brace,amplitude=5pt,mirror},xshift=0.4pt,yshift=-0.4pt](6,0) -- (9,0) node[black,midway,yshift=-.4cm] {\footnotesize $c$};
			
			\end{tikzpicture}}\ . We can then write the formula of $g_{\lambda}$ in terms of $g_{\lambda}[x,y,z]$.
		
		\begin{theo}\label{Dun4} Let $\lambda = (3^a2^b1^c)$. Then 
			$$
			g_{\lambda}=\sum_{i=0}^{b}g_{\lambda}[a{+}i,b{-}i,c]+\sum_{i=1}^{c}g_{\lambda}[a,b,c{-}i].
			$$
		\end{theo}
		
		\begin{proof}
			The first term $\sum_{i=0}^{b}g_{\lambda}[a{+}i,b{-}i,c]$ sums over all the cases in \fref{1}(a) and the second term $\sum_{i=0}^{c}g_{\lambda}[a,b,c{-}i]$ sums over all the cases in \fref{1}(b).
		\end{proof}
		
		Thus, we have a complete recursion for $g_\lambda$. The recursive formula for $g_\lambda$ not only shows that $g_\lambda$ is Schur-positive in $q,t$-analogs, also gives us a way of writing $g_\lambda$ into $q,t$-analogs and powers of $(qt)$.
For example, suppose $\lambda = 1^4$. Then $\lambda' =(4)$ so that 
taking into account the possible numbers of $0$'s in a tableau $T \in SSYT((4),012)$, 
we see that 
$$g_{(1^4)} = g_\lambda[0,0,0]+g_\lambda[0,0,1]+g_\lambda[0,0,2]+g_\lambda[0,0,3]+g_\lambda[0,0,4].$$
But $g_\lambda[0,0,0]=g_\lambda[0,0,1]=0$. 
We can apply Theorem \ref{Dun1} to compute 
$$
g_\lambda[0,0,2] = 
\sum_{i=0}^0 (qt)^i [2{-}i{-}\lfloor (i{+}1)/2 \rfloor \rightarrow 4{-}2 {-}3i]_{q,t} =[2]_{q,t},
$$
\begin{eqnarray*}
g_\lambda[0,0,3] &=& 
\sum_{i=0}^1 (qt)^i [3{-}i{-}\lfloor (i{+}1)/2 \rfloor \rightarrow 6{-}2 {-}3i]_{q,t}\\
&=& [3\rightarrow 4] +(qt)[1\rightarrow 1] = [3]_{q,t}+[4]_{q,t} +qt,
\end{eqnarray*}
and
\begin{eqnarray*}
g_\lambda[0,0,4] &=& 
\sum_{i=0}^1 (qt)^i [4{-}i{-}\lfloor (i{+}1)/2 \rfloor \rightarrow 8{-}2 {-}3i]_{q,t}\\
&=& [4\rightarrow 6] +(qt)[2\rightarrow 3] \\
&=& [4]_{q,t}+[5]_{q,t} +[6]_{q,t}+qt([2]_{q,t}+[3]_{q,t}),
\end{eqnarray*}
Thus 
$$g_{(1^4)} = [2]_{q,t}+[3]_{q,t} +
2 [4]_{q,t}+[5]_{q,t} +[6]_{q,t}+qt(1+[2]_{q,t}+[3]_{q,t}).$$

In general, we see that 
$$\langle \Delta_{e_2} e_n[X],e_n[X]\rangle =
\sum_{s=2}^n g_\lambda[0,0,s].$$
We claim that $g_\lambda[0,0,n]$ is a $q,t$-analogue of $2\binom{n+1}{3}$. 
To see this, we shall use a formula of \cite{HRW} to show 
that $$\langle \Delta_{e_2} e_n[X],e_n[X]\rangle|_{q=t=1} =
2\binom{n+2}{4}$$ from which it follows 
that 
\begin{eqnarray*} 
g_\lambda[0,0,n]|_{q=t=1} &=& 
\langle \Delta_{e_2} e_n[X],e_n[X]\rangle|_{q=t=1} -\langle \Delta_{e_2} e_{n-1}[X],e_{n-1}[X]\rangle|_{q=t=1} \\
&=& 2 \binom{n+2}{4}- 2\binom{n+1}{4} = 2\binom{n+1}{3}. 
\end{eqnarray*}
It is proved in  \cite{HRW} that 
\begin{equation}
\Delta_{e_k} e_n[X]|_{t=1/q} = \frac{q^{\binom{k}{2}-k(n-1)}}{[k+1]_q} 
\qbinom{n}{k} e_n[X(1{+}q{+}\cdots {+}q^k)].
\end{equation}
Repeatedly applying the sum rule that 
$$s_{\lambda}[X+Y]= \sum_{\mu \subseteq \lambda} s_\mu[X]s_{\lambda/\mu}[Y],$$
we see that 
\begin{eqnarray*}
\Delta_{e_k} e_n[X]|_{t=1/q}
&=& \frac{q^{\binom{k}{2}-k(n-1)}}{[k+1]_q} \qbinom{n}{k}
\sum_{\overset{i_s \geq 0}{i_0+i_1+ \cdots +i_k =n}} 
\prod_{s=0}^k e_{i_s}[q^s X] \\ 
&=& \frac{q^{\binom{k}{2}-k(n-1)}}{[k+1]_q}  \qbinom{n}{k}
\sum_{\overset{i_s \geq 0}{i_0+i_1+ \cdots +i_k =n}} 
\prod_{s=0}^k q^{si_s}e_{i_s}[X].
\end{eqnarray*}
It follows that 
\begin{equation*}
\langle \Delta_{e_k} e_n[X],e_n[X]\rangle|_{t=1/q} = 
\frac{q^{\binom{k}{2}-k(n-1)}}{[k+1]_q}  \qbinom{n}{k}
\sum_{\overset{i_s \geq 0}{i_0+i_1+ \cdots +i_k =n}} 
\prod_{s=0}^k q^{si_s}.
\end{equation*}
But it is easy to see that 
$$\sum_{\overset{i_s \geq 0}{i_0+i_1+ \cdots +i_k =n}} 
\prod_{s=0}^k q^{si_s} = \qbinom{n+k}{k}$$ since the LHS is just 
the sum of $q^{|\lambda|}$ over all partitions $\lambda$ contained in 
the $n \times k$ rectangle. 
Thus 
\begin{equation}\label{Dekenen}
\langle \Delta_{e_k} e_n[X],e_n[X]\rangle|_{t=1/q}
= \frac{q^{\binom{k}{2}-k(n-1)}}{[k+1]_q}  \qbinom{n}{k} \qbinom{n+k}{k}.
\end{equation}

Setting $q=1$ and $k =2$ in (\ref{Dekenen}), we see that 
\begin{eqnarray*}
\langle \Delta_{e_2} e_n[X],e_n[X]\rangle|_{q=t=1} &=& 
\frac{1}{3}\frac{n!}{2!(n-2)!}\frac{(n+2)!}{2!n!} \\
&=& \frac{4}{2}\frac{(n+2)!}{4!(n-2)!} = 2 \binom{n+2}{4}.
\end{eqnarray*}

Next consider $\lambda =(1^2,2)$ so that $\lambda' = (1,3)$. In this case, 
we can classify the tableau $T \in SSYT((1,3),012)$ by whether 
the bottom corner square contains a 1, in which case we get a 
term $g_{\lambda}[1,0,2]$, or the bottom corner square contains a 0, in which 
case we get a contribution of $g_{\lambda}[0,1,2]$, $g_{\lambda}[0,1,1]$, or 
$g_{\lambda}[0,1,0]$, depending on the number of $0$'s in the first row. 
But $g_{\lambda}[0,1,0] = g_{\lambda}[0,0,1]=0$. By Theorem 
\ref{Dun3},
\begin{eqnarray*}
g_{\lambda}[0,1,1]&=& \sum_{i=0}^1 g_{\lambda}[i,0,2{-}2i] \\
&=& g_{\lambda}[0,0,2] + g_{\lambda}[1,0,0] \\
&=& [2]_{q,t} +g_{\lambda}[1,0,0],
\end{eqnarray*} 
and 
\begin{eqnarray*}
g_{\lambda}[0,1,2] &=& \sum_{i=0}^1 g_{\lambda}[i,0,3{-}2i] \\
&=& g_{\lambda}[0,0,3] + g_{\lambda}[1,0,1] \\
&=& [3]_{q,t} +[4]_{q,t} +qt +g_{\lambda}[1,0,1].
\end{eqnarray*} 
By Theorem \ref{Dun2}
\begin{eqnarray*}
g_{\lambda}[1,0,0] &=& (qt) g_{\lambda}[0,0,0]+[3\rightarrow 3] = [1]_{q,t},\\
g_{\lambda}[1,0,1] &=& (qt) g_{\lambda}[0,0,1]+[4\rightarrow 5] = [2]_{q,t}+[3]_{q,t}, 
\ \mbox{and} \\
g_{\lambda}[1,0,2]&=& (qt) g_{\lambda}[0,0,2]+[3\rightarrow 5] \\
&=&  (qt)[2]_{q,t}+[3]_{q,t}+[4]_{q,t}+[5]_{q,t}. 
\end{eqnarray*}
It follows that 
$$g_{1^2,2} = [1]_{q,t}+2[2]_{q,t}+3[3]_{q,t}+2[4]_{q,t}+
[5]_{q,t}+(qt)(1+[2]_{q,t}).$$

\subsection{The relation between the combinatorial proof and direct computation}

We show in this subsection that the combinatorial involution of the enriched tableaux implies the cancellation step of the computation of $g_\lambda[a,k_1,k_2]$. 

Firstly, we present the case $g_\lambda[0,0,k]$. We illustrate the relation by an example of $k=5$. Let $\lambda=(5)$ and suppose there are no $0$'s in the filling. Then the contribution of all tableaux of this form is $g_\lambda[0,0,5]$. 

If there are $i$ $1$'s in the filled Young diagram, then there are $5{-}i$ $2$'s. \tref{Dun1} shows that when we sum over all cases classified by number of $1$'s, we have 
\begin{eqnarray*}
g_\lambda[0,0,5]&=&\sum_{i=0}^{k}w(i,k{-}i)\\
&=&-\qtn{4}-(qt)^2\qtn{1}+(qt)^2\qtn{2}+(qt)\qtn{5}+\qtn{8}\\
&&-(qt)\qtn{2}+(qt)^2\qtn{1}+(qt)\qtn{4}+\qtn{7}\\
&&+0+(qt)\qtn{3}+\qtn{6}\\
&&+(qt)\qtn{2}+\qtn{5}\\
&&+\qtn{4}\\
&=&\qtn{5\rightarrow 8}+(qt)\qtn{3\rightarrow 5}+(qt)^2\qtn{2}.
\end{eqnarray*}

Notice that there is a cancellation in the last step of the equation. The cancellation cancels terms of different signs in the last steps, which follows the same idea of the combinatorial proof. The injection of the combinatorial proof maps the negative terms into the positive terms, giving this cancellation. \taref{Dun1} shows all enriched tableaux of shape $\lambda=(5)$ with no $0$'s and their corresponding weight. From this, we can see that the first column and the first two rows of the second column are canceled, leaving only the red terms. These give $g_\lambda[0,0,5]=\qtn{5\rightarrow 8}+(qt)\qtn{3\rightarrow 5}+(qt)^2\qtn{2}$.

\begin{table}[ht!]
	\centering
	\begin{tabular}{|c|c|c|c|c|c|}
		\hline
		\# of $1$'s & no $\overline{2}$ & one $\overline{2}$ & two $\overline{2}$ & three $\overline{2}$ & four $\overline{2}$ \\\hline
		$0$ &\tabw{0}{4}{0}{-\qtn{4}} & \tabw{0}{3}{1}{-(qt)^2\qtn{1}} & \tabwr{0}{2}{2}{(qt)^2\qtn{2}} & \tabwr{0}{1}{3}{(qt)\qtn{5}} & \tabwr{0}{0}{4}{\qtn{8}} \\\hline
		$1$ &\tabw{1}{3}{0}{-(qt)\qtn{2}} & \tabw{1}{2}{1}{(qt)^2\qtn{1}} & \tabwr{1}{1}{2}{(qt)\qtn{4}} & \tabwr{1}{0}{3}{\qtn{7}} & \\\hline
		$2$ &\tabw{2}{2}{0}{0} & \tabwr{2}{1}{1}{(qt)\qtn{3}} & \tabwr{2}{0}{2}{\qtn{6}} &  & \\\hline
		$3$ &\tabw{3}{1}{0}{(qt)\qtn{2}} & \tabwr{3}{0}{1}{\qtn{5}} &&&\\\hline
		$4$ &\tabw{4}{0}{0}{\qtn{4}} & &&& \\\hline
	\end{tabular}
	\caption{Contribution of tableaux of shape $\lambda=(5)$}
	\label{table:Dun1}
\end{table}

Next, we show this relation for the case $g_{\lambda}[a,0,k]$. The semi-standard Young tableaux contributing to such $g_{\lambda}$ contain 2 parts -- the first part has $a$ columns of $2$ rows of which the bottom row is filled with $1$'s, and the second part has $k$ columns of $1$ row filled with some $1$'s and $2$'s, look like $T=$
\raisebox{-13pt}{\begin{tikzpicture}[scale =.35]
\draw[help lines] (6,0) rectangle (9,2);
\draw[help lines] (6,0) rectangle (12,1);
\fillll{7}{1}{1}\fillll{8}{1}{\cdots}\fillll{9}{1}{1}\fillll{10}{1}{1}\fillll{11}{1}{\cdots}\fillll{12}{1}{2}
\fillll{7}{2}{2}\fillll{8}{2}{\cdots}\fillll{9}{2}{2}
\draw [thick, blue,decorate,decoration={brace,amplitude=5pt,mirror},xshift=0.4pt,yshift=-0.4pt](6,0) -- (9,0) node[black,midway,yshift=-.4cm] {\footnotesize $a$};
\draw [thick, blue,decorate,decoration={brace,amplitude=5pt,mirror},xshift=0.4pt,yshift=-0.4pt](9,0) -- (12,0) node[black,midway,yshift=-.4cm] {\footnotesize $k$};
\end{tikzpicture}}.

In the proof of \tref{Dun2}, we have the recursion that $g_{\lambda}[a,0,k] =(qt)^ag_{\lambda}[0,0,k] + \sum_{j=0}^{a-1}(qt)^j\sum_{i=0}^{k}\qtn{2k{-}i{+}3a{-}3j{-}2}$. We want to show the recursion combinatorially. If we fill all cells of the top row of the first $a$ columns of $T$ with $2$, then the contribution of the first $a$ columns is $(qt)^a$, and the contribution of the last $k$ columns is $g_{\lambda}[0,0,k]$ for the same reason as the first case of $g_{\lambda}[0,0,k]$, except that the last $k$ columns cannot be all $1$'s as there should be a $\hat 2$ in these columns. Actually this exceptional restriction about filling does not affect $g_{\lambda}[0,0,k]$ as an all $1$'s filling case contributes $0$ to $g_{\lambda}[0,0,k]$. Otherwise, if there is at least one $\bar 2$ or $\hat 2$ in the first $a$ columns, then there must be no $2$ in the last $k$ columns. Suppose there are $i$ $1$'s in the last $k$ columns and $j$ $2$'s in the first $a$ columns. Then the contribution is $(qt)^j \qtn{2k{-}i{+}3a{-}3j{-}2}$. These are all fixed points in the involution described in the combinatorial proof, and we now have found the implication of combinatorial involution in the cancellation of $g_{\lambda}[a,0,k]$.

Finally, for the case $g_\lambda[a,k_1,k_2]$, the combinatorics is straightforward in the recursion in \tref{Dun4}, thus we see the connection of the combinatorial proof and the direct computation.

\section{Proofs by Generating Functions\label{s-gen-fun}}
Here we illustrate two proofs using generating functions. They are not different in nature.

\noindent
First generating function proof.

It is clear that $g_{\lambda}=0$ unless $\lambda'$ has at most $3$ parts, i.e., $\lambda'=(a{+}b{+}c,b{+}c,c)$ for $a,b,c\ge 0$.
The idea is to show the generating function
$$G(u_1,u_2,u_3)=\sum_{a,b,c\ge 0} g_{(a{+}b{+}c,b{+}c,c)'} u_1^a u_2^b u_3^c $$
has only nonnegative coefficients.

Firstly, we use the quotient formula for Schur functions:
$$ s_{a{+}b{+}c,b{+}c,c}[x+y+z]=\frac{1}{(x{-}y)(y{-}z)(x{-}z)} \det \begin{pmatrix}
                                                           x^{a+b+c+2} & x^{b+c+1} & x^c \\
                                                           y^{a+b+c+2} & y^{b+c+1} & y^{c} \\
                                                           z^{a+b+c+2} & z^{b+c+1} & z^{c} \\
                                                         \end{pmatrix}.
$$
Next, from the view of MacMahon partition analysis (see, e.g., \cite{MPA1}, \cite{xiniterate}), $G$ is easily seen to be a rational power series. 
Here we only need the following fact:

If $\gamma_{ij}\ge 0$ for all $i,j$, then
\begin{multline*}
  \sum_{a,b,c\ge 0} x^{\gamma_{11}a+\gamma_{12} b+\gamma_{13} c}
y^{\gamma_{21}a+\gamma_{22} b+\gamma_{23} c}
z^{\gamma_{31}a+\gamma_{32} b+\gamma_{33} c}
 u_1^a u_2^b u_3^c \\
 =  \frac{1}{(1-x^{\gamma_{11}} y^{\gamma_{21}} z^{\gamma_{31}}u_1)(1-x^{\gamma_{12}} y^{\gamma_{22}} z^{\gamma_{32}}u_2)
 (1-x^{\gamma_{13}} y^{\gamma_{23}} z^{\gamma_{33}}u_3)}.
\end{multline*}

One simple case will illustrate the idea. By the quotient formula,
$$ s_{a+b+c,b+c,c}[1+q+t]=\frac{1}{(1{-}q)(1{-}t)(q{-}t)}(t^{a+b+c+2} q^{b+c+1} + \text{other terms}),$$
where the ``other terms'' are the five terms of similar type obtained by expanding the determinant.
Now we have
$$  \sum_{a,b,c\ge 0} R(q,t) t^{a+b+c} q^{b+c} u_1^a u_2^b u_3^c = R(q,t) \frac{1}{(1-tu_1)(1-u_2 qt)(1-u_3qt)},$$
where $R(q,t)=t^2 q/((1{-}q)(1{-}t)(q{-}t))$ is a rational function independent of $a,b,c$.

Thus we can write $G$ as a sum of $6\times 3=18$ rational functions. This can be carried out by Maple and we normalize to obtain
$$ G= \frac{P}{(1-m_1)(1-m_2)\cdots (1-m_{15})},$$
where $m_i$ are monomials, and $P$ is a polynomial with $1023$ terms.
Through a complicated search procedure, we found a decomposition
$ G= \sum_{i=1}^{27} Q_i$ where each $Q_i$ is easily seen to have only nonnegative coefficients. For instance, one of the terms is
$$Q_1=  {\frac {u_{{1}}u_{{2}}q \left( {q}^{3}u_{{2}}{+}t \right) }{ \left( qtu_
{{2}}{-}1 \right)  \left( {t}^{2}u_{{1}}{-}1 \right)  \left( u_{{1}}{-}1
 \right)  \left( qu_{{2}}{-}1 \right)  \left( qtu_{{3}}{-}1 \right)
 \left( qu_{{1}}{-}1 \right)  \left( {q}^{2}u_{{2}}{-}1 \right)  \left( {q
}^{3}u_{{2}}{-}1 \right) }}.
  $$
This proves that $G$ has only nonnegative coefficients and hence $g_\lambda \in \Z_{\ge 0}[q,t]$.
As a proof, we only need to verify that these $Q_i$'s sum to $G$ (which is routine by computer) but not how to find them. 
We are not going to explain in detail how to decompose $G$ since the idea is not mature.

\medskip
\noindent
Second generating function proof.

After the first proof was obtained, Professor Adriano Garsia investigated some data of $g_\lambda$ and conjectured that
$g_\lambda$ is indeed also Schur positive in $q,t$-analogs. More precisely, we have
$$ g_\lambda = \sum_{i>j\ge 0 } b_{i,j} s_{(i-1,j)}[q{+}t] =  \sum_{i>j\ge 0} b_{i,j} q^j t^j[i{-}j{-}1]_{q,t}, $$
where $b_{i,j}$ are nonnegative for all $i>j\ge 0$.

This is equivalent to writing
$(t-q)G = F - \tau F $, where
$$F= \sum_{i>j \ge 0}  b_{i,j}(u_1,u_2,u_3) t^i q^j ,$$
and showing the nonnegativity of $b_{i,j}(u_1,u_2,u_3)$.
To obtain an explicit formula of $F$ from $G$, it is better to make the change of variable by $q=\bar q/t$. Then
\begin{align*}
  F&= \sum_{i>j \ge 0}  b_{i,j}(u_1,u_2,u_3) t^{i-j} \bar q^j ,\\
  \tau F &= \sum_{i>j \ge 0}  b_{i,j}(u_1,u_2,u_3) q^i t^j = \sum_{i>j \ge 0}  b_{i,j}(u_1,u_2,u_3) t^{-(i-j)} \bar q^i.
\end{align*}
It follows that $F$ consists of all terms in the series expansion of $(t-\bar q/t)G$ with positive powers in $t$. This can be realized by the following constant term
$$ F = (z {-}\bar q/z) G \big|_{t=z}  \sum_{k\ge 1} (t/z)^k \Big|_{z^0}= (z {-}\bar q/z) G \big|_{t=z}  \frac{t/z}{1-t/z} \Big|_{z^0}.$$
Thus $F$ can be calculated by MacMahon's partition analysis techniques.

The complexity of $G$ suggests that this approach does not work for $\Delta_{e_3}e_n[X]$, so we go over the computation and the use of 
Lemma \ref{l3} which is the point of departure for the other proofs.

Using the explicit formula of $F$, which has 132 terms in the numerator and 11 factors in the denominator, we are able to decompose $F$ as a sum of $7$ rational functions that are easily seen to have nonnegative coefficients:
\begin{multline*}
F=  -{\frac { \left( qt{u_{{1}}}^{3}+qt{u_{{1}}}^{2}u_{{2}}+qtu_{{1}}{u_{{
2}}}^{2}+t{u_{{1}}}^{2}+tu_{{1}}u_{{2}}+t{u_{{2}}}^{2} \right) t}{
 \left( u_{{1}}-1 \right)  \left( tu_{{2}}-1 \right)  \left( {t}^{3}u_
{{3}}-1 \right)  \left( {t}^{2}u_{{1}}-1 \right)  \left( {t}^{3}u_{{2}
}-1 \right)  \left( {q}^{2}{t}^{2}{u_{{1}}}^{3}-1 \right)  \left( {q}^
{2}{t}^{2}{u_{{2}}}^{3}-1 \right) }}\\
-{\frac {{u_{{2}}}^{3}t \left( {t}
^{3}+qt \right) }{ \left( u_{{1}}-1 \right)  \left( tu_{{2}}-1
 \right)  \left( {t}^{3}u_{{3}}-1 \right)  \left( {t}^{2}u_{{1}}-1
 \right)  \left( {t}^{3}u_{{2}}-1 \right)  \left( {q}^{2}{t}^{2}{u_{{2
}}}^{3}-1 \right)  \left( {t}^{2}u_{{2}}-1 \right) }}\\
-{\frac {u_{{2}}t
 \left( {q}^{2}{t}^{4}{u_{{1}}}^{2}{u_{{2}}}^{2}+{q}^{2}{t}^{2}{u_{{1}
}}^{2}u_{{2}}+{q}^{2}{t}^{2}u_{{1}}{u_{{2}}}^{2}+q{t}^{2}{u_{{1}}}^{2}
+q{t}^{2}u_{{1}}u_{{2}}+q{t}^{2}{u_{{2}}}^{2}+{t}^{2}u_{{1}}+{t}^{2}u_
{{2}}+1 \right) }{ \left( u_{{1}}-1 \right)  \left( tu_{{2}}-1
 \right)  \left( {t}^{3}u_{{3}}-1 \right)  \left( {t}^{3}u_{{2}}-1
 \right)  \left( {q}^{2}{t}^{2}{u_{{1}}}^{3}-1 \right)  \left( {q}^{2}
{t}^{2}{u_{{2}}}^{3}-1 \right)  \left( qtu_{{2}}-1 \right) }}\\
-{\frac {
 \left( qtu_{{1}}+t \right) {t}^{2}{u_{{1}}}^{3}}{ \left( {q}^{2}{t}^{
2}{u_{{1}}}^{3}-1 \right)  \left( {t}^{2}u_{{1}}-1 \right)  \left( {t}
^{3}u_{{3}}-1 \right)  \left( tu_{{2}}-1 \right)  \left( qtu_{{2}}-1
 \right)  \left( u_{{1}}-1 \right)  \left( tu_{{1}}-1 \right) }}\\
 -{
\frac {u_{{1}}u_{{2}}{t}^{2}}{ \left( {t}^{2}u_{{1}}-1 \right)
 \left( {t}^{3}u_{{3}}-1 \right)  \left( tu_{{2}}-1 \right)  \left( qt
u_{{2}}-1 \right)  \left( u_{{1}}-1 \right)  \left( tu_{{1}}-1
 \right)  \left( {t}^{3}u_{{2}}-1 \right) }}\\
 -{\frac {u_{{1}}{t}^{2}u_{
{3}} \left( q{t}^{2}{u_{{1}}}^{2}+{t}^{2}u_{{1}}+1 \right) }{ \left( u
_{{1}}-1 \right)  \left( qtu_{{2}}-1 \right)  \left( tu_{{2}}-1
 \right)  \left( {t}^{3}u_{{3}}-1 \right)  \left( {q}^{2}{t}^{2}{u_{{1
}}}^{3}-1 \right)  \left( qtu_{{3}}-1 \right)  \left( tu_{{1}}-1
 \right) }}\\
 -{\frac { \left( {q}^{2}{t}^{4}{u_{{1}}}^{2}{u_{{2}}}^{2}+{
q}^{2}{t}^{2}{u_{{1}}}^{2}u_{{2}}+{q}^{2}{t}^{2}u_{{1}}{u_{{2}}}^{2}+q
{t}^{2}{u_{{1}}}^{2}+q{t}^{2}u_{{1}}u_{{2}}+q{t}^{2}{u_{{2}}}^{2}+{t}^
{2}u_{{1}}+{t}^{2}u_{{2}}+1 \right) tu_{{3}}}{ \left( u_{{1}}-1
 \right)  \left( qtu_{{2}}-1 \right)  \left( tu_{{2}}-1 \right)
 \left( {t}^{3}u_{{3}}-1 \right)  \left( {q}^{2}{t}^{2}{u_{{1}}}^{3}-1
 \right)  \left( {q}^{2}{t}^{2}{u_{{2}}}^{3}-1 \right)  \left( qtu_{{3
}}-1 \right) }}.
\end{multline*}
This may be treated as our second proof, but in the same vein of our first proof.

\section{The $\Delta_{e_3}$ Case\label{s-e3}}

For the $\Delta_{e_3}e_n[X]$ case, we have a similar formula
\begin{multline*}
 e_4[X]=\\
{\frac {\TH_{{4}}[X;q,t]}{ \left( q{-}t \right)  \left( {q}^{2}{-}t \right)
 \left( {q}^{3}{-}t \right) }}-{\frac { \left( {q}^{2}{+}q{+}t{+}1 \right) \TH_{{3,1}[X;q,t]}}{ \left( q{+}t \right)  \left( {q}^{3}{-}t \right)
 \left( q{-}t \right) ^{2}}}-{\frac { \left( qt{-}1 \right) \TH_{{2,2
}}[X;q,t]}{ \left( -{t}^{2}{+}q \right)  \left( {q}^{2}{-}t \right)  \left( q{-}t
 \right) ^{2}}}\\+{\frac { \left( {t}^{2}{+}q{+}t{+}1 \right) \TH_{{2,1,1
}}[X;q,t]}{ \left( q{+}t \right)  \left( -{t}^{3}{+}q \right)  \left( q{-}t
 \right) ^{2}}}-{\frac {\TH_{{1,1,1,1}}[X;q,t]}{ \left( q{-}t \right)
 \left( -{t}^{3}{+}q \right)  \left( -{t}^{2}{+}q \right) }}.
\end{multline*}

This gives
\begin{multline*}
    \scalar{ \Delta_{e_3} e_n[X]}{ s_{\lambda'}}  = \\
{\frac {s_{\lambda}[B_4]}{ \left( q{-}t \right)  \left( {q}^{2}{-}t \right)
 \left( {q}^{3}{-}t \right) }}-{\frac { \left( {q}^{2}{+}q{+}t{+}1 \right) s_{\lambda}[B_{3,1}]}{ \left( q{+}t \right)  \left( {q}^{3}{-}t \right)
 \left( q{-}t \right) ^{2}}}-{\frac { \left( qt{-}1 \right) s_\lambda[B_{{2,2
}}]}{ \left( -{t}^{2}{+}q \right)  \left( {q}^{2}{-}t \right)  \left( q{-}t
 \right) ^{2}}} \\ + {\frac { \left( {t}^{2}{+}q{+}t{+}1 \right) s_\lambda[B_{{2,1,1
}}]}{ \left( q{+}t \right)  \left( -{t}^{3}{+}q \right)  \left( q{-}t
 \right) ^{2}}}-{\frac {s_\lambda[B_{{1,1,1,1}}]}{ \left( q{-}t \right)
 \left( -{t}^{3}{+}q \right)  \left( -{t}^{2}{+}q \right) }}.
\end{multline*}

By playing with partial fraction decompositions, the best formula we
 have is
\begin{align}
  \scalar{ \Delta_{e_3} e_n[X]}{ s_{\lambda'}}  &= \frac{F_\lambda(q,t)- F_\lambda(t,q)}{q-t}
  {-\frac {s_{\lambda}[1{+}q{+}t{+}q^2]/q^2-s_\lambda[1{+}q{+}t{+}t^2]/t^2}{2 \left( {q}^{2}{-}{t}^{2} \right) }},
\end{align}
where $F_\lambda=F_\lambda(q,t)$ is given by
\begin{multline*}
  F_\lambda ={\frac {s_\lambda[1{+}q{+}q^2{+}q^3]-s_\lambda[1{+}q{+}t{+}qt]}{ \left( q{-}1 \right) {q}^{2} \left( {q}^{2}{-}t \right) }}-{
\frac {s_\lambda[1{+}q{+}q^2{+}q^3]-s_{\lambda}[1{+}q{+}t{+}q^2]}{{q}^{2} \left( q{-}1 \right)  \left( {q}^{3}{-}t \right) }}\\
-{
\frac {(q{+}1)(s_{\lambda}[1{+}q{+}t{+}q^2]-s_\lambda[1{+}q{+}t{+}qt])}{ 2\left( q{-}t \right) {q}^{2} \left( q{-}1 \right) }}
+{\frac {s_\lambda[1{+}q{+}t{+}qt]}{2{q}^{2}t}}.
\end{multline*}

One can use this formula to prove 
 that $g_\lambda$ is a polynomial divided by $(1{-}q)$.
Nevertheless, it clear that this approach becomes more and more complicated 
so that the proof of the general $\Delta_{e_d}e_n[X]$ case seems to require 
a new idea.

\end{document}